\documentclass[12pt]{amsart}
\usepackage{amssymb}
\usepackage{amsmath}
\usepackage{amsthm}
\usepackage{calc}
\usepackage{verbatim}
\usepackage{hyperref}
\theoremstyle{plain}
\newtheorem{thm}{Theorem}[section]
\newtheorem{lem}[thm]{Lemma}
\newtheorem{cor}[thm]{Corollary}
\newtheorem{rem}[thm]{Remark}
\newtheorem{eg}[thm]{Example}
\newtheorem{prop}[thm]{Proposition}

\theoremstyle{definition}
\newtheorem{defn}[thm]{Definition}

\newcounter{probno}
\stepcounter{probno}

\newcommand{\proofof}[1]{\medskip\noindent{\bf Proof of #1. }}
\renewcommand{\qed}{\hfill$\Box$ \par \medskip}

\newcommand{\id}{\mathrm{id}}

\newcommand{\norm}[2][{}]{\left\|#2\right\|_{#1}}

\newcommand{\origin}{\mathbf{0}}

\newcommand{\R}{\mathbf{R}}

\newcommand{\C}{\mathbf{C}}
\newcommand{\cp}{\mathbf{P}}
\newcommand{\Z}{\mathbf{Z}}
\newcommand{\N}{\mathbf{N}}

\newcommand{\Q}{\mathbf{Q}}

\DeclareMathOperator{\aut}{Aut}

\DeclareMathOperator{\ord}{ord}
\newcommand{\poles}{\mathcal{P}_{polar}}
\newcommand{\simples}{\mathcal{P}_{simple}}
\newcommand{\action}{f_\sharp}
\newcommand{\primes}{\mathcal{P}}

\DeclareMathOperator{\pic}{Pic}
\renewcommand{\div}{\mathop{\mathrm{Div}}}
\DeclareMathOperator{\supp}{supp}
\DeclareMathOperator{\crit}{crit}
\newcommand{\destab}{\mathcal{D}}
\newcommand{\polepts}{\mathcal{Z}}
\newcommand{\feta}{f_\eta}
\newcommand{\orig}{\mathop{\mathrm{orig}}}
\newcommand{\simplepolepts}{\mathcal{Z}_{simple}}
\DeclareMathOperator{\res}{res}
\newcommand{\val}{\mathcal{V}}
\newcommand{\symp}{Symp}
\newcommand{\divplus}{{\div(D^+)}}

\DeclareMathOperator{\fix}{Fix}
\def\fan{\Delta}
\newcommand{\Id}{\mathrm{Id}}

\theoremstyle{plain}
\newtheorem*{thma}{Theorem A}
\newcommand{\zeradicate}{{A}}
\newtheorem*{propb}{Proposition B}
\newcommand{\tricotomy}{{B}}
\newcommand{\sbe}{{C}}
\newtheorem*{thmc}{Theorem C}
\newcommand{\corrig}{{D}}
\newtheorem*{thmd}{Theorem D}
\newcommand{\toricact}{{E}}
\newtheorem*{thme}{Theorem E}
\newcommand{\toriccor}{{F}}
\newtheorem*{thmf}{Theorem F}

\newcommand{\kod}{\mathop{\mathrm{kod}}}

\title{Rational surface maps with invariant meromorphic two-forms}

\date{\today}

\author{Jeffrey Diller}
\address{Department of Mathematics\\
         University of Notre Dame\\
         Notre Dame, IN 46556}
\email{diller.1@nd.edu}

\author{Jan-Li Lin}
\address{Department of Mathematics\\
         University of Notre Dame\\
         Notre Dame, IN 46556}
\email{jlin4@nd.edu}

\thanks{This work was supported in part by National Science Foundation grant DMS--1066978.}
\begin{document}

\begin{abstract}
Let $f:S\to S$ be a rational self-map of a smooth complex projective surface $S$ and $\eta$ be a meromorphic two-form on $S$ satisfying $f^*\eta = \delta\eta$ for some $\delta\in\C^*$.  We show that under a mild topological assumption on $f$, there is a birational change of domain $\psi: X\to S$ such that $\psi^*\eta$ has no zeros.  In this context, we investigate the notion of algebraic stability for $f$, proving that $f$ can be made algebraically stable if and only if it acts nicely on the poles of $\eta$.  We illustrate this last result in the case $\eta = \frac{dx\wedge dy}{xy}$, where we translate our stability result into a condition on whether a circle homeomorphism associated to $f$ has rational rotation number.
\end{abstract}

\maketitle
\setcounter{tocdepth}{1}
\tableofcontents
\markboth{\today}{\today}

\section{Introduction}
\label{introduction}

Much recent research concerns the dynamics of meromorphic maps on compact complex surfaces.  Many of the examples that have guided this work are distinguished by, among other things, the fact that they come equipped with invariant meromorphic two-forms.  Perhaps most notable among these are plane polynomial automorphisms, which preserve the Euclidean form $dx\wedge dy$, and monomial maps which preserve the form $\frac{dx\wedge dy}{xy}$.  Other significant examples, though by no means all of them, occur in \cite{BeDi1, BeKi, Bu1, Ca3, DJS, Mc}.

In this paper, we undertake a more systematic study of surface maps with invariant two-forms.  Specifically, we let $S$ be a smooth complex projective surface, $\eta$ a meromorphic two-form on $S$, and $f:S\to S$ a rational map.    We emphasize that there is a finite `indeterminacy' set $I(f)$ on which $f$ cannot be continuously defined.  The pullback $f^*\eta$ is nevertheless a meromorphic two-form, defined pointwise on $S\setminus I(f)$, and by e.g. Hartog's extension across $I(f)$.  We say that $f$ \emph{preserves} $\eta$ if $f^*\eta = \delta\eta$ for some $\delta = \delta(f,\eta)\in\C^*$.  The condition $\delta\neq 0$ implies that $f$ is dominant, i.e. the image $\overline{f(S\setminus I(f))}$ is all of $S$ rather than some proper subvariety.   

Observe that preservation of a two-form persists under birational conjugacy.  That is, if $f$ preserves $\eta$ and $\psi:\tilde S \to S$ is birational, then $\psi^{-1}\circ f\circ \psi:\tilde S\to \tilde S$ preserves $\psi^*\eta$.   So we can use birational change of coordinate to try to put $\eta$ into a simple normal form.   Examples, including monomial maps and plane polynomial automorphisms, suggest that we aim to make $\psi^*\eta$ zero-free, i.e. $\div\psi^*\eta \leq 0$.

On the other hand, $\eta$ might not be unique.  That is, $f$ might preserve a pair of linearly independent two-forms $\eta_1,\eta_2$, in which case $f$ preserves the fibration of $S$ defined by the non-trivial rational function $R:=\eta_1/\eta_2$.  Specifically, $R\circ f = cR$ for some $c\in\C^*$, so for any $n\in\Z$, the form $\eta = R^n\,\eta_2$ will be preserved by $f$, and when $n>1$ the zeros of $\eta$ will include those of $R$.  We will show in Proposition \ref{badeg} that such zeros cannot always be eliminated by birational change of coordinate.  These circumstances are, however, atypical.  For instance, since $f$ factors through the linear map $z\to cz$ on $\cp^1$, the topological degree $\lambda_2$ of $f$ is the same as the so-called `first dynamical degree' $\lambda_1$ (defined more precisely in \S \ref{background}), which measures the asymptotic rate at which iterates of $f$ expand the area of a typical curve in $S$.  The inequality $\lambda_1\neq\lambda_2$ is therefore a reasonable way to exclude the problem of non-uniqueness.

\begin{thma}
Suppose that $S$ is a smooth complex projective surface and $f:S\to S$ is a rational map preserving a meromorphic two-form $\eta$ on $S$.   If $\lambda_1\neq \lambda_2$, then there is a birational map $\psi:\tilde S \to S$ such that $\psi^*\eta$ has no zeros.   
\end{thma}

As will become apparent in \S \ref{egs}, our main results Theorems {\zeradicate}, {\sbe } and ${\corrig }$ are significant mostly when the Kodaira dimension of $S$ is $-\infty$.  That is, $S$ is a complex projective surface birationally equivalent to $\cp^1\times B$ for some Riemann surface $B$.  The case when $S$ is rational (i.e. $B = \cp^1$) is particularly interesting.  Here we have three mutually exclusive possibilities for the two-form $\psi^*\eta$ in the conclusion of Theorem {\zeradicate }.

\begin{propb}
If $\eta$ is a zero-free meromorphic two-form on a rational surface $S$, then there is a birational map $\psi:\cp^2 \to S$ such that one of the following holds.
\begin{itemize}
\item (smooth case) The divisor $-\div\psi^*\eta$ is a smooth cubic curve.
\item (toric case) $\psi^*\eta = \frac{dx\wedge dy}{xy}$.
\item (euclidean case) $\psi^*\eta = dx\wedge dy$.
\end{itemize}
\end{propb}

\noindent A similar result (Lemma \ref{dicotomy}) with two instead of three possible cases holds when $S$ is irrational, but in that situation we can (Theorem \ref{irrationalegs}) characterize the maps as well as the forms.

In the rest of the introduction we will consider how the existence of a zero-free invariant two-form $\eta$ might assist in understanding dynamics of $f$.  Specifically, we consider `algebraic stability' for a rational map $f:S\to S$ preserving $\eta$.  Any rational self-map induces a natural pullback $f^*:\div(S)\to \div(S)$ on divisors, and this descends to a finite dimensional operator $f^*:\pic(S)\to\pic(S)$ on linear equivalence classes.  The map $f$ is \emph{algebraically stable} if pullback iterates well, i.e. $(f^n)^* = (f^*)^n$ for all $n\in\N$.   Algebraic stability implies that the first dynamical degree of $f$ is the largest eigenvalue of $f^*$.  This has been a necessary starting point for analytic constructions (see e.g. \cite{Sib, Gue}) used to study the ergodic theory of $f$.

Forn{\ae}ss and Sibony \cite{FoSi} (see Proposition \ref{geometric criterion} in \cite{FoSi}) observed that algebraic stability fails precisely when some irreducible curve $C\subset S$ is contracted by $f$ and then mapped by a further iterate $f^n$ to a point in the indeterminacy set of $f$.  We call such curves $C$ \emph{destabilizing} for $f$.  Diller and Favre \cite{DiFa} showed that destabilizing curves can sometimes be eliminated by blowing up their forward images: when $f:S\to S$ is birational (with or without an invariant two-form), there always exists a modification (i.e. a composition of point blowups) $\pi:X\to S$ that lifts $f$ to an algebraically stable map $f_X:X\to X$.
On the other hand, Favre \cite{Fa2} observed that this result does not extend to non-invertible rational maps.  He showed that it fails in particular for certain monomial maps, given in affine coordinates by $(x,y)\mapsto (x^ay^b,x^cy^d)$ for $a,b,c,d\in\Z$.  Monomial maps preserve the two-form $\frac{dx\wedge dy}{xy}$, so preserving a two-form does not automatically guarantee that one can arrange for algebraic stability.   Nevertheless, the invariant two-form allows one to somewhat restrict the search for a stable model $X\to S$.  The key fact here is that while $f$ can contract curves which are not poles of $\eta$, the image of a contracted curve is always contained in the divisor of $\eta$.  From this we first establish

\begin{thmc}
Suppose that $f:S\to S$ is a rational map preserving a meromorphic two-form $\eta$ such that $-\div\eta$ is effective.
If there exists a modification $\pi:X\to S$ lifting $f$ to an algebraically stable map, then one can choose $\pi$ so that $-\div\pi^*\eta$ is effective.
\end{thmc}

In light of Theorem~{\sbe}, it makes sense to introduce the following terminology.

\begin{defn} Let $\eta$ be a meromorphic two-form on $S$ with $\div\eta<0$.  A modification $\pi_X:X\to S$ is an \emph{elaboration} of $\eta$ if $\div\pi_X^*\eta<0$, too.  Let $f:S\to S$ be a rational map that preserves $\eta$.  We say that $f$ is \emph{corrigible} if for any elaboration $\pi_X:X\to S$, there is a further elaboration $\pi:Y \to X$ lifting $f$ to an algebraically stable map.  We say that $f$ is \emph{corrigible along $\eta$} if for any elaboration $\pi_X:X\to S$ there exists a further elaboration $\pi: Y\to X$ such that no pole of $(\pi_X\circ\pi)^*\eta$ is destabilizing for the lift of $f$ to $Y$.
\end{defn}

Theorem {\sbe} and the Forn{\ae}ss-Sibony criterion for algebraic stability imply that if $f$ is corrigible, then it is corrigible along $\eta$.  Our next result is that the converse holds.

\begin{thmd}
Let $f:\cp^2\to\cp^2$ be a rational map preserving a zero-free meromorphic two-form $\eta$.  Then $f$ is corrigible if and only if $f$ is corrigible along $\eta$.
\end{thmd}

The gain here is that for each irreducible component $C$ of $\div\eta$ and any $n\in\N$, the strict transform $f^n(C)$ is contained in another irreducible component of $\div\eta$.  Hence it is far easier to check whether poles of $\eta$ destabilize $f$ than it is to check whether other curves do so.  We stress, however, that Theorem {\corrig} does \emph{not} say that only poles of $\eta$ can destabilize $f$.  The proof makes clear that one needs to elaborate much more to eliminate all destabilizing curves than one does to eliminate only destabilizing poles.

Corrigibility along $\eta$ turns out to be automatic in the smooth case of Proposition \tricotomy.  The toric case $\eta = \frac{dx\wedge dy}{xy}$ is more interesting, and we consider it at length.  Our approach, in the spirit of \cite{Ca1, BFJ, FaWu} is to consider at once all poles of all lifts of $\eta$ in all possible elaborations.  Modulo an equivalence identifying the `same' pole in different elaborations, this is a countable set $\poles$ naturally identified (via the theory of toric surfaces) with the set of rational rays in $\R^2$.   We show that

\begin{thme}
Suppose that $f:\cp^2\to\cp^2$ preserves $\frac{dx\wedge dy}{xy}$.  Then $f$ induces a map $\action:\poles\to\poles$ which corresponds to a piecewise linear map $T_f:\R^2\to\R^2$ such that $T_f(\Z^2)\subset \Z^2$ and $T_f:\R^2\setminus\{\origin\}\to\R^2\setminus\{\origin\}$ is a finite covering.
\end{thme}

For $f$ birational, $T_f$ is a `piecewise linear automorphism of $\Z^2$', i.e. a homeomorphism satisfying $T_f(\Z^2)=\Z^2$.  In fact, $T_f$ is a homeomorphism even in all non-invertible examples that we know.  Hence in these cases, we obtain a circle homeomorphism $T_f:S_1\to S_1$ by letting $T_f$ act on one dimensional rays.  The key invariant for dynamics of the induced circle homeomorphism is its \emph{rotation number}, which we use with Theorem {\corrig} to determine corrigibility.

\begin{thmf}
Suppose in Theorem {\toricact} that $T_f$ is a homeomorphism.
\begin{itemize}
\item If $T_f$ is orientation reversing, then $f^2$ is corrigible.
\item If $T_f$ is orientation preserving with rotation number $m/n\in\Q$, then $f^n$ is corrigible.
\item If the rotation number of $T_f$ is irrational, then no iterate of $f$ is birationally conjugate to any algebraically stable map.
\end{itemize}
\end{thmf}

This result is well-known for monomial maps \cite{Fa2, JW} but new for non-monomial maps preserving $\eta$.  In particular, it implies that many non-monomial maps are not corrigible.  Restricting to birational maps, it has another interesting consequence.  Usnich \cite{Us} (see also Blanc's work \cite{Bl1} in this direction) showed that any piecewise linear automorphism $T$ of $\Z^2$ is equal to $T_f$ for some birational map $f$ preserving $\frac{dx\wedge dy}{xy}$.  Hence the result from \cite{DiFa} that birational maps are corrigible implies, via Theorems {\sbe} and {\toricact}, that any circle map induced by a piecewise linear automorphism of $\Z^2$ must have rational rotation number (see Corollary \ref{plaz2}).  This was originally proven in a rather indirect fashion by Ghys and Sergiescu \cite{GhSe} with later and more direct proofs given by Liousse \cite{Li} and Calegari \cite{Cal}.

We have little to say here about corrigibility along $\eta$ in the Euclidean case $\eta = dx\wedge dy$.  One can again consider the action of $f$ on the set $\poles$ of all poles of all elaborations of $\eta$, but this time $\poles$ is uncountable, and the situation is similar to one considered by Favre and Jonsson in their work \cite{FaJo1} on dynamical degrees for polynomial maps of $\C^2$.  Their approach is to regard (an analogue of) $\poles$ as a dense subset in a certain space of valuations, and then to exploit the tree structure of the valuation space in order to understand algebraic stability for polynomial maps $f:\C^2\to\C^2$.  One might hope to mimic their approach to say more about corrigibility of rational maps $f$ preserving $dx\wedge dy$.  It is worth pointing out, however, that if the Jacobian conjecture fails on $\C^2$ then any results about maps preserving $dx\wedge dy$ must apply to the counterexamples.



Theorems {\zeradicate}, {\sbe}, and {\corrig } and their proofs remain true on any compact K\"ahler surface $S$.  However, As we will explain at the end of \S \ref{egs}, they are fairly uninteresting for non-projective surfaces.  A referee for this paper pointed out that the same results extend to projective surfaces over any algebraically closed field of characteristic zero.  

The authors would like to thank Serge Cantat for pointing out some good examples of non-invertible maps that preserve smooth and Euclidean two-forms, and William Gignac for patiently explaining ideas from \cite{FaJo1}.  
Finally, they are grateful to the aforementioned referee for his many constructive suggestions for improving the paper.

\section{Rational maps and meromorphic two-forms}
\label{background}

From now on, unless otherwise noted, the word `surface' will mean `smooth complex projective surface'.  We begin with some background concerning rational maps on surfaces. The articles \cite{DiFa, DDG1} have longer discussions (and more proofs) in the same spirit.

Let $X,Y$ be smooth complex projective surfaces and $f:X\to Y$ be a dominant rational map.  That is, there exists a surface $\Gamma$, a modification (i.e. composition of point blowups) $\pi_1:\Gamma\to X$ and a surjective holomorphic map $\pi_2:\Gamma\to Y$ such that $f= \pi_2\circ \pi_1^{-1}$.  The composition is well-defined, except at a finite set $I(f)\subset X$ of points that are images of complex curves contracted by $\pi_1$.  We call $\Gamma$ the \emph{graph} of $f$, though technically it is a desingularization of the graph and as such not unique.  We call $I(f)$ the indeterminacy set of $f$.

Throughout this paper, the word `curve' will always be used in the set theoretic sense.  That is, it will always refer to a (reduced and possibly reducible) compact one dimensional complex subvariety of $X$.
We adopt the convention for $p\in I(f)$ that $f(p)$ is the connected complex curve $\pi_2\circ\pi_1^{-1}(p)$.  On the other hand, for curves $C\subset X$ we adopt the (somewhat opposite) convention that $f(C) = \overline{f(C\setminus I(f))}$ is the set-theoretic strict transform of $C$.  Hence $f(C)$ is a point if and only if $f$ contracts $C$.   When $f$ does not contract $C$, we write $m(f,C)$ to denote the local multiplicity of $f$ about $C$, and $\deg f|_C$ to denote the topological degree of the restriction $f:C\to f(C)$.

Let $\div(X)$ denote the set of divisors on $X$.  For any $D\in\div(X)$ and any irreducible curve $C\subset X$, we write $\ord(C,D) $ to denote the order of $C$ in $D$.   Similarly we write $\ord(p,D) := \sum_C \mu(p,C)\ord(C,D)$, where the sum is over all irreducible curves $C\subset X$ and $\mu(p,C)$ is the multiplicity of $C$ at $p$, i.e. the order of vanishing at $p$ of a local defining function for $C$.

There are natural linear maps $f^*:\div(Y)\to \div(X)$, $f_*:\div(X)\to\div(Y)$.  If $C\subset X$ and $C'\subset Y$ are irreducible curves such that $f(C) = C'$, then we have in particular that $f_*C = (\deg f|_C) C'$ and $\ord(C,f^*C') = m(f,C)$.   The critical divisor $\crit(f)\in\div(C)$ is defined by
\begin{equation}
\label{hurwitz}
\crit(f) = \div f^*\eta - f^*\div\eta
\end{equation}
where $\eta$ is some/any meromorphic two-form on $X$ and $\div\eta$ is the associated divisor.  The irreducible components of $\crit(f)$ come in two types: curves that are contracted by $f$, and `branch' curves $C$ for which $m(f,C) >1$.  In the latter case, we have $\ord(C,\crit(f)) = m(f,C)-1$, a fact which may be restated as follows.

\begin{prop}
\label{jformula}
Suppose that $C\subset X$ is not contracted by $f$ and that $\eta$ is any meromorphic two-form on $Y$.  Then
$$
\ord(C,f^*\eta) + 1 = m(f,C) (\ord(f(C),\eta) + 1).
$$
\end{prop}

\noindent Note that here and elsewhere $\ord(C,\eta)$ is shorthand for $\ord(C,\div \eta)$.  The quantity $\ord(C,\eta) + 1$ is often called the \emph{log-discrepancy} of $C$ relative to $\eta$.  The monograph \cite{FaJo3} uses the term \emph{thinness} for a closely related quantity associated to a valuation.

Pushforward and pullback of divisors are compatible with linear equivalence and so descend to corresponding maps between Picard groups $\pic(X)$ and $\pic(Y)$.  The latter are in turn compatible (via assignment of Chern classes) with maps $f_*,f^*$ between $H^{1,1}(X)$ and $H^{1,1}(Y)$.

Pushforward and pullback are adjoint with respect to the intersection product, i.e. $f^*\alpha\cdot \beta=\alpha\cdot f_*\beta$, but they are not in general inverses.  More precisely, one has the following `pushpull' formula.

\begin{thm}[\cite{DiFa}, Theorem 3.3]
\label{pushpull}
For any divisor $D\in\div(X)$,
$$
f_*f^* D = \lambda_2(D) + E^-(D),
$$
where $E^-:\div(X)\to\div(X)$ is a linear operator such that
\begin{itemize}
 \item $E^-(D)$ depends only on the Chern class of $D$,
 \item $\supp E^-(D) \subset f(I(f))$,
 \item $D\cdot E^-(D)\geq 0$, and
 \item $E^-(D)=0$ is equivalent to $D\cdot E^-(D)=0$ and to $D\cdot f(p)=0$ for all $p\in I(f)$.
\end{itemize}
\end{thm}

It is important to note that pushforward and pullback do not necessarily behave well under composition of rational maps and in particular under iteration of self-maps.

\begin{defn}
A rational self-map $f:S\to S$ is \emph{algebraically stable} if $(f^*)^n = (f^n)^*$ for all $n\in\N$.
\end{defn}

\noindent It does not matter in the definition whether one considers pullback of divisors, linear equivalence classes, or cohomology classes.  It is equivalent to require that $(f_*)^n = (f^n)_*$ for all $n\in\N$.  More geometrically, one has

\begin{prop}[see \cite{FoSi} page 139 and \cite{DiFa} Theorem 1.14]
\label{geometric criterion}
$f:S\to S$ is algebraically stable if and only if there is no irreducible curve $C$ contracted by $f$ such that $f^n(C)\in I(f)$ for some $n>0$.
\end{prop}

\noindent If such a curve $C$ exists, and $n>0$ is the smallest integer such that $f^n(C)\in I(f)$, then we will call both $C$ and the orbit segment $f(C),\dots, f^n(C)$ \emph{destabilizing} for $f$.

We let $\lambda_2$ denote the topological degree of $f:S\to S$.  This is equal to the number of preimages of a generic point or, alternatively, the multiplier for the one dimensional operator $f^*:H^{2,2}(S)\to H^{2,2}(S)$ induced by $f$ on the top cohomology of $S$.  One can associate a similar quantity $\lambda_1$ to pullback of divisors.

\begin{defn}
The \emph{first dynamical degree} of a rational map $f:S\to S$ is the quantity
$
\lambda_1 = \lambda_1(f) := \lim_{n\to\infty} \norm{(f^n)^*}^{1/n}
$,
where $f^*$ is pullback on $H^{1,1}(S)$.
\end{defn}

As the definition suggests, the limit on the right side always exists and is independent of choice of the norm $\norm{\cdot}$.  The topological and first dynamical degrees always satisfy the basic inequality $\lambda_1^2 \geq \lambda_2$.  When $S$ is projective, one has
$$
\lambda_1 = \sup_{D,D'\in\div(S)} \lim_{n\to\infty} (f^{n*}D\cdot D')^{1/n},
$$
where `$\cdot$' denotes intersection product.  The supremum is achieved whenever $D^2,(D')^2 > 0$ and in particular when $D,D'$ are ample.

If $f$ is algebraically stable, $\lambda_1$ is just the magnitude of the largest eigenvalue $r_1$ of $f^*:H^{1,1}(S)\to H^{1,1}(S)$.  In general, $\lambda_1 \leq r_1$ and there are nef classes (i.e. limits of K\"ahler classes) $\theta^*,\theta_*\in H^{1,1}(S)$ such that $f^*\theta^* = r_1\theta^*$ and $f_*\theta_* = r_1\theta_*$.  These are unique up to positive multiple if $r_1^2 > \lambda_2$, but in any case, the fact that $f^*$ and $f_*$ are adjoint relative to intersection product has an elementary consequence that will be useful to us.

\begin{prop}
\label{perp}
For any $\theta\in H^{1,1}(S)$, $\lim_{n\to\infty} r_1^{-n} (f^*)^n\theta = 0$ implies that $\theta\cdot\theta_* = 0$.
\end{prop}

This fact is especially useful to us when combined with the Hodge Index Theorem, which we use in the following slightly non-standard form (see \cite{DJS}).

\begin{thm}
\label{hodge}
Let $\theta\in H_\R^{1,1}(S)$ be a non-trivial nef class, and $D\in\div(S)$ an effective divisor.  Set $\div(D)$ equal to the set of divisors $D'$ supported on the irreducible components of $D$.  Then
$\theta\cdot D = 0$ implies that $(D')^2\leq 0$ for all $D'\in\div(D)$.  Moreover, either
\begin{itemize}
\item $(D')^2 < 0$ for all $D'\neq 0$; or
\item there exists an effective divisor $D'\in\div(D)$ whose Chern class is a positive multiple of $\theta$.
\end{itemize}
In the second case, any divisor $D''\in\div(D)$ satisfying $(D'')^2 = 0$ has the property that $k'D'$ is linearly equivalent to $k''D''$ for some non-zero integers $k',k''$.
\end{thm}

\section{Classification and Examples}
\label{egs}

Complex surfaces are classified into a number of different families, but not all of them are equal for our purposes.  Here we will give several particular examples of rational maps $f$ that preserve two-forms $\eta$ on various types of surfaces $S$.  We will also argue that in most cases the possibilities for invariant two-forms and for maps that preserve them are rather restricted.  In brief, maps of interest are far more plentiful when $S$ is rational, i.e. birationally equivalent to $\cp^2$.  For other types of surfaces $S$, a rational map $f:S\to S$ automatically preserves some other geometric feature (a canonical fibration or holomorphic two-form) of $S$, and this limits the set of maps to which the results of this paper will apply.
 
We will rely heavily here on the classification of compact complex surfaces as described in \cite[\S VI.1]{BHPV}.  Throughout, we let $\kod(S)\in\{-\infty,0,1,2\}$ denote the Kodaira dimension of $S$.  

\subsection{Maps on rational surfaces.}
Up to conjugacy, the only non-trivial rational functions $f:\cp^1\to \cp^1$ that preserve meromorphic one-forms are affine maps $z\mapsto az+b$ which preserve $dz$, and power maps $z\mapsto z^k$ which preserve $dz/z$.  By contrast there are many rational maps $f:\cp^2\to\cp^2$ that preserve (zero-free) two-forms $\eta$.  Theorem {\zeradicate } and Proposition {\tricotomy }, however, tell us that there are only three main possibilities for the form $\eta$.  We begin by proving the latter.

\proofof{Proposition {\tricotomy}}
Since contracting curves in $S$ will not create zeros for $\eta$, we may suppose that $S$ is minimal; i.e. if $S\neq \cp^2$, then $S$ is either $\cp^1\times \cp^1$ or a Hirzebruch surface.  That is, $S$ is fibered by smooth rational curves, and there is a smooth rational curve $E$ transverse to this fibration with $E^2 \leq 0$.  Since $K_S$ is not a multiple of $E$, it follows that $\supp\div\eta$ is not contained in $E$.  So if $E^2<0$, then $E$ is unique and we choose a point $p \in \supp\div\eta\setminus E$, blow it up and contract the strict transform of the fiber that contained $p$.  This reduces $E^2$ by one and, since $\ord(p,\eta)<0$, creates no zeros for $\eta$.  We may therefore repeat this process until
$E^2=0$, in which case $S=\cp^1\times \cp^1$ admits two transverse fibrations by smooth rational curves.  Then we blow up any $p\in\supp\div\eta$ and contract the strict transforms of both fibers that containing $p$.  The resulting surface is $\cp^2$, and the resulting form has no zeros.

If now $-\div\eta$ is not smooth and reduced, then $-\div\eta$ is a singular cubic curve representing one of eight isomorphism classes.  These fall into two subfamilies, according to whether some/any singular point for $-\div\eta$ has order two or three.  We claim that in the first case one can birationally change coordinate so that $\eta = \frac{dx\wedge dy}{xy}$, and in the second, one can change coordinate so that
$\eta = dx\wedge dy$.

We indicate how to do this in one case and leave the reader to puzzle out the rest.  Suppose $-\div\eta = C+L$, where $C$ is a smooth conic and $L$ is a line meeting $C$ in two distinct points.  By blowing up the two intersections and one other point on $C$ and then contracting the three lines joining these points, one obtains a quadratic birational map $\psi:\cp^2\to\cp^2$ such that $-\div\psi^*\eta$ is a sum of three lines in general position.  A further linear transformation moves these lines to the $x$ and $y$ axes and the line at infinity, and a final scaling of coordinates transforms $\eta$ to $\frac{dx\wedge dy}{xy}$.
\qed

In the remainder of this subsection, we illustrate each of the three cases in Proposition {\tricotomy } with various examples.  We also give an example illustrating the necessity of the hypothesis $\lambda_1\neq\lambda_2$ in Theorem {\zeradicate}.

\subsubsection{The smooth case}  Suppose $\div\eta = -C$ where $C\subset\cp^2$ is a smooth cubic curve.  One obtains some birational maps $f:\cp^2\to\cp^2$ preserving $\eta$ as follows (see \cite{Di1} for more details).  Let $p_0, p_1, p_2\in C$ be three distinct, non-collinear 
points of $C$.  Let $q:\cp^2\to\cp^2$ be the (quadratic) birational transformation that blows up each point $p_j$ and collapses each of the three lines that contain two of the $p_j$.  Then $q(C)$ is a cubic curve isomorphic to $C$, and it follows that there is a linear map $T\in\aut(\cp^2)$ such that $T(q(C)) = C$.  Taking $f:=T\circ q$, we have $f(C) = C$ and furthermore that $f$ preserves $\eta$ with $\div\eta = -C$.  By varying the points $p_j$ and taking compositions of the resulting maps $f$, one obtains all (see \cite{Pan}) birational $f$ preserving $\eta$, and many of these maps have complicated dynamics.

We thank Serge Cantat for explaining to us how to get some non-invertible examples in the smooth case.  Let $\pi:X\to \cp^2$ be the blowup of $\cp^2$ at nine distinct points obtained by intersecting with some other cubic curve $C'\subset\cp^2$.  Then the pencil of cubics determined by $C$ and $C'$ in $\cp^2$ lifts (by strict transform) to an elliptic fibration of $X$, and the nine irreducible curves $E_1,\dots, E_9$ contracted by $\pi$ are sections of the fibration.  Let $f_X:X\to X$ be the rational map whose restriction to a general fiber $F$ is given by $(p - E_1|_F) \mapsto 2(p-E_1|_F)$.  Here we are taking advantage of the isomorphism between $F$ and the set $\pic^0(F) \subset \pic(F)$ of classes of degree 0 divisors on $F$.  Clearly $\lambda_2(f) = 4$, and $f$ preserves the elliptic fibration of $X$.  In fact, $f_X$ is holomorphic and non-critical in a neighborhood of any smooth fiber (e.g. the strict transform of $C$).  It follows that $f_X^*\pi^*\eta = 4\pi^*\eta$.  Hence $f_X$ descends to a rational 
map $f:\
\cp^2\to\cp^2$ preserving $\eta$.  By changing the cubic $C'$ used to define $f$, one gets other examples which may be composed with each other or with birational maps that preserve $\eta$.

\subsubsection{The toric case}
\label{section_example_toric}
Suppose (in affine coordinates) $\eta = \frac{dx\wedge dy}{x y}$, so that $C = -\div\eta$ is a union of three lines.  Any monomial map $f:(x,y)\mapsto (x^ay^b,x^cy^d)$, with $a,b,c,d\in\Z$, satisfies $f^*\eta = (ad-bc)\eta$.  As it happens the topological degree of $f$ is $\lambda_2=|ad-bc|$.  Hence $f$ is birational if and only if $f^*\eta = \pm\eta$. Note also that any monomial map restricts to an unbranched self-cover $(\C^*)^2\to(\C^*)^2$ on the complement of $\supp\div\eta$

One obtains other birational maps $f$ preserving $\eta$ as in the smooth case above, the only additional consideration being that each irreducible component of $C$ contains exactly one of the three points $p_j$.  In this case, one is led to the very simple formula (in homogeneous coordinates)
\begin{equation}
\label{idaction}
f([x_0:x_1:x_2]) = [x_0\ell_0 : x_1\ell_1 : x_2\ell_2]
\end{equation}
where $\ell_j:\C^3\to \C$ are linear forms chosen so that for $j\neq k$, the lines $\{\ell_j=0\}$ and $\{x_k=0\}$ meet at $p_k$.
For instance, the lines $\{\ell_0=0\}$, $\{\ell_1=0\}$ and $\{x_2=0\}$ meet at a common point $p_2$, so we can write
$\ell_1 = a \ell_0 + b x_2$ for some $a,b\in\C^*$. We can also write $\ell_2 = c \ell_0 + d x_1$ for some $c,d\in\C^*$ for the same reason.
One can then see that $f$ contracts each line $\{\ell_0=0\}$ to the point $[0:b:d]\in \{x_0=0\}$.
Similarly, we conclude that $f$ contracts the line $\{\ell_j=0\}$ to some point on $\{x_j=0\}$.
Therefore, unlike monomial maps, these $f$ contract lines $\{\ell_j=0\}$ that are not poles of $\eta$.  On the other hand, $f$ maps each pole $\{x_j=0\}$ of $\eta$ to itself by a linear transformation, and we will see more generally below that if $\pi:X\to \cp^2$ is any elaboration of $\eta$, the lift $f_X:X\to X$ of $f$ to $X$ preserves all irreducible components of $\div\pi^*\eta_X$.

A still different birational map preserving $\eta$ is given in affine coordinates by $f:(x,y)\mapsto (y,\frac{1+y}{x})$.  One checks easily that $f^5 = \id$, and that $f$ lifts to an automorphism of a rational surface obtained by blowing up two of the three double points of $C$.  This map appears in the preprint \cite{Us} of Usnich, which presents conjectures later proved by Blanc \cite{Bl1} about the group $\symp$ of plane Cremona transformations preserving $\frac{dx\wedge dy}{xy}$.

Note that for all these toric examples, and therefore for any map $f$ in the semigroup they generate, one has $f^*\eta = \delta\eta$ where $\delta = \pm\lambda_2$.  We will prove below that $\delta$ must be an integer dividing $\lambda_2$ for \emph{any} rational map preserving $\frac{dx\wedge dy}{xy}$, but we do not know whether $|\delta|$ must actually equal $\lambda_2$ in general.

\subsubsection{The Euclidean case}  Finally we discuss some rational maps that preserve $\eta=dx\wedge dy$.  The leading example here is that of polynomial automorphisms $f:\C^2\to\C^2$ whose dynamics have been considered in many papers (e.g. \cite{BeSm1,BLS,FoSi2}).  Any such $f$ extends to a birational map on $\cp^2$ satisfying $f^*\eta = \delta\eta$, and $\delta$ can be any non-zero complex number.

One can get more birational maps preserving $dx\wedge dy$ by the methods indicated for the smooth case.  As we explained in the proof of Proposition {\tricotomy }, this form is birationally equivalent to the form $\tilde \eta = \frac{dx\wedge dy}{y-x^3}$ with $-\div\tilde\eta$ equal to a cuspidal cubic.  The latter is more convenient here, because one can then construct birational maps preserving $\tilde \eta$ exactly as in the smooth case and conjugate them back to birational maps preserving $dx\wedge dy$.  These include many maps (see \cite{BeKi, Mc}) which lift to dynamically non-trivial (i.e. positive entropy) automorphisms on some modification $\pi:X\to\cp^2$.  Since the first dynamical degree $\lambda_1$ is an invariant of birational conjugacy, always equal to an integer for a polynomial automorphism of $\C^2$, and never equal to an integer for a rational surface automorphism with positive entropy, one sees that not all birational maps preserving 
$dx\wedge dy$ come from polynomial automorphisms.

One can also obtain non-invertible maps preserving $\eta$ as we did in the smooth case, but Serge Cantat pointed out to us that non-invertible examples can be written down directly: for any non-constant rational function $A:\cp^1\to\cp^1$, the map
$$
f(x,y) = (A(x),\delta y/A'(x))
$$
satisfies $f^* dx\wedge dy = \delta dx\wedge dy$ and has topological degree equal to that of $A$.

\subsubsection{Maps that preserve more than one two-form}
Let us finally illustrate the necessity of the hypothesis $\lambda_1\neq \lambda_2$ in Theorem {{\zeradicate}} by way of a very simple and particular example.

\begin{prop}
\label{badeg}
Let $f:\cp^2\to\cp^2$ be given by $f(x,y) = (2x,y^2)$.  Then $f$ preserves $\eta := \frac{x^n\,dx\wedge dy}{y}$ for any $n\in\N$.  But for any rational surface $X$ and any birational $\psi:X\to \cp^2$, the
form $\psi^*\eta$ has zeros if $n>0$.
\end{prop}

\begin{proof}
Let $\Gamma$ be the graph of $\psi$ and $\pi_1:\Gamma\to X$, $\pi_2:\Gamma\to \cp^2$ be projections onto domain and range.  
Suppose for the moment that $\pi_2$ does not contract any curve to the unique point $q_\infty \in \{x=0\}\cap \C^2$ at infinity along $\{x=0\}$.  We have $\ord(p,\eta)\geq 0$ for every $p\in \{x=0\}\cap\C^2$, including the $(0,0)= \{x=0\}\cap\{y=0\}$.  Hence it follows from Lemma \ref{orderdecreases} below that $\pi_2^*\eta$ has a zero along every irreducible component of $\pi_2^*\{x=0\}$.  Therefore the same lemma tells us that if $\psi^*\eta$ has no zeros, then $\pi_1$ contracts every irreducible component of $\pi_2^*\{x=0\}$.  Since $\pi_2^*\{x=0\}$ has positive self-intersection (equal to that of $\{x=0\}$), and $\pi_1$ is a modification, this is impossible.

Now suppose instead that $\pi_2^{-1}(q_\infty)$ is one dimensional, i.e. $\pi_2 = \sigma\circ \tilde\pi_2$ factors through the blowup $\sigma:\hat{\cp^2}\to\cp^2$ of $\cp^2$ at $q_\infty$.  Then $\sigma^*\eta$ has a zero of order $n$ along the strict transform $L$ of $\{x=0\}$ and, as one can see from direct computation, a simple pole along 
$\sigma^{-1}(q_\infty)$.  In particular, $\ord(p,\sigma^*\eta) \geq 0$ at \emph{every} point in $L$.  Since $L^2 = 0$ is non-negative, the argument from the previous paragraph again implies that $\psi^*\eta$ must have zeroes.
\end{proof}

Hence it is not always possible to eliminate the zeros of an invariant two-form using birational coordinate change.  The problem, at least in this instance, may be viewed as one of uniqueness.  That is, while we cannot eliminate the zeros of $\eta$, there is a another invariant two-form $\frac{dx\wedge dy}{xy}$ which is zero-free, and the difference between the divisors of the forms is supported on lines that generate an $f$-invariant pencil in $\cp^2$.

\subsection{Irrational surfaces with negative Kodaira dimension}
\label{minusinf}

Now suppose that $S$ is an irrational surface with $\kod(S) = -\infty$.  Then $S$ is birationally equivalent to $B\times\cp^1$, where $B$ is a smooth curve of positive genus, and the projection $\phi:(x,y)\to x$ onto $B$ is the Albanese fibration.  Hence, any rational self-map $f$ of $B\times \cp^1$ preserves fibers of $\pi$ and induces a holomorphic map $\check f:B\to B$ on the base.  When $g:=\mathrm{genus}\,B >1$, the base map $\check f$ is a finite order automorphism, and it follows (see \cite[Lemma 4.1]{DDG1}) that $\lambda_1(f) = \lambda_2(f)$, regardless of whether or not $f$ preserves a meromorphic two-form, so Theorem {\zeradicate } is vacuous unless $g=1$.

Concerning the conclusion of Theorem {\zeradicate }, it turns out that we can describe, up to birational conjugacy, all rational self-maps of $S$ that preserve a zero-free two-form $\eta$.  The key part is to put $\eta$ into a normal form, parallel to Proposition {\tricotomy }.  Recall \cite[V.2]{Har} that any minimal surface birationally equivalent to $B\times \cp^1$ is a $\cp^1$-bundle $\phi:S\to B$.  As with Hirzebruch surfaces there is a section $\sigma$ of this bundle with $\sigma^2 \leq 0$ minimal.  We have $S=B\times\cp^1$ if and only if $\sigma^2 =0$, in which case one can take $\sigma$ to be any horizontal slice $\{y=const\}$.  When $\sigma^2<0$, the section $\sigma$ is unique.  In any case, any divisor on $S$ is numerically equivalent to an integer combination of $\sigma$ and any fiber $V$ of $\phi$.  In particular the genus formula applied to $\sigma$ and $V$ implies that $$
K_S \sim -2\sigma + (2g-2+\sigma^2)V,
$$
where $K_S$ is the canonical class on $S$ and $\sim$ denotes numerical equivalence. 

\begin{lem}
\label{dicotomy}
Let $S$ be an irrational surface with $\kod(S) = -\infty$ and $\eta$ be a zero-free meromorphic two-form on $S$.  Then after a birational change of coordinates, we may suppose that $S$ is a $\cp^1$-bundle over a curve $B$ with $g:=\mathrm{genus}\,B \ge 1$ and either
\begin{itemize}
 \item $g>1$ and $\div \eta = -2\sigma$, where $\sigma$ is the unique section with minimal self-intersection $\sigma^2<0$;
 \item $g=1$, $S = B\times \cp^1$ is a product, and $\div\eta = -H_1 - H_2$, where the $H_j = \{y=y_j\}$ are (possibly equal) horizontal slices.
\end{itemize}
\end{lem}

\begin{proof}
As in the proof of Proposition {\tricotomy } we may suppose that $S$ is minimal. Given a zero-free two-form $\eta$, we write 
$$
-\div\eta = a\sigma + D_V + D
$$ 
where $a\geq 0$, $D_V$ is supported on a finite union of fibers of $\phi$, and $D$ is an effective divisor transverse to both $\sigma$ and to fibers.
Now if $D=0$, then we can further eliminate the components $D_V$ as follows: blow up a point $p\in \sigma\cap \supp D_V$ and contract the fiber $V$ of $\phi$ incident to $p$.  The result is a birational map $\psi:\tilde S \to S$ from a different $\cp^1$-bundle $\tilde S\to B$ in which the fiber $V$ is replaced by the exceptional curve $\tilde V$ obtained by blowing up $p$.  By Hurwitz formula 
$$
\ord(\tilde V,\psi^*\eta) = \ord(p,\eta) + 1 = \ord(V,\eta) + \ord(\sigma,\eta) + 1 = \ord(V,\eta) - 1.
$$
Hence we have eliminated one component (counting multiplicity) of $D_V$.  After repeating finitely many times, $D_V = 0$. So $-\div\eta = -2\sigma$.
Using the numerical equivalence formula for $K_S$ above, we find that $\sigma^2 = 2-2g$.  That is, $\sigma^2 = 0$ if and only if $g=1$.  Hence the lemma is true under the assumption $D=0$.

If on the other hand $D>0$, then $D$ is not numerically equivalent to a multiple of $V$, and we must have
$$
0 < D\cdot V = (-K_S-D_V-a\sigma)\cdot V = (2-a).
$$
So $a\in\{0,1\}$. Similarly,
$$
0 \leq D\cdot \sigma = -\sigma^2(a-1) - (2g-2) -\sigma\cdot D_V,
$$
where all terms on the right are non-positive and hence are all zero.  Thus $D$ can be non-trivial only if $D\cdot\sigma = 0$, which in turn implies that $g=1$ and $D_V =0$.  If $a=0$, then we must also have $\sigma^2 = 0$.  Thus $S = B\times \cp^1$ is a product and $D\sim 2\sigma$ is supported on a pair of horizontal curves $\{y=const\}$, i.e. the second conclusion of the Lemma holds.  

If $a=1$ then $D\cdot V = 1$ and $D\cdot\sigma = 0$ imply that $D$ is a section of $\phi$ that does not meet $\sigma$.  Since $g=1$, we have 
$$
0 = K_S^2 = (\sigma + D)^2 = \sigma^2 + D^2,
$$
so $D^2 = -\sigma^2$.  If $\sigma^2<0$, then we can blow up a point $p\in \supp D$ and contract the vertical fiber incident to $p$, obtaining a birational map $\psi:\tilde S\to S$ where $\tilde S\to B$ is another $\cp^1$ bundle in which $\div\psi^*\eta = \psi^{-1}(\sigma) + \psi^{-1}(D)$ and $\psi^{-1}(D), \psi^{-1}(\sigma)$ are disjoint sections whose self-intersections have decreased/increased by $1$, respectively. By repeating this proceedure finitely many times, we arrange that $\sigma^2= D^2 = \sigma\cdot D = 0$, so that $S=B\times \cp^1$ and $\sigma$ and $D$ are horizontal slices.
\end{proof}

\begin{thm}
\label{irrationalegs}
Suppose that $S$ is a complex projective surface birationally equivalent to $B\times \cp^1$ for some smooth curve $B$ with positive genus $g$.  Suppose $f:S\to S$ be a rational map preserving a zero-free meromorphic two-form $\eta$.  If $g>1$, then $f$ is birational and $\lambda_1(f) = 1$.  If $g = 1$ then we may conjugate so that $S = B\times \cp^1$ and either
\begin{itemize}
\item (irrational Euclidean case) $\eta = dx\wedge dy$ and $f(x,y) = (\check f(x), \alpha (x)y + \beta (x))$ for some holomorphic functions $\alpha,\beta:B \to\cp^1$; or
\item (irrational smooth case) $\eta = dx \wedge \frac{dy}{y}$ and $f(x,y) = (\check f(x), \alpha(x) y^k)$ for some $k\neq 0$ and holomorphic $\alpha:B \to \cp^1$.
\end{itemize}
Note that we use $dx$ here to denote the standard non-vanishing one-form on the genus 1 curve $B$ thought of as a quotient of $\C$.
\end{thm}

One checks easily that any map of one of the two forms given in the case $g=1$ will preserve the corresponding two-form 
$\eta$.  In particular, when $\eta = dy/y$, we have $\lambda_2(f) = k(\deg \check d)$ and $\lambda_1(f) = \max\{k,\check d\}$ where $\check d$ is the topological degree of the base map $\check f$.  So if $\check d>1$, it follows that $\lambda_1\neq \lambda_2$ and Theorem {\zeradicate } may be used to detect maps birationally conjugate to the ones described here.

\begin{proof}
Suppose first that $g=1$.  From the lemma, we may suppose that $S = B\times \cp^1$ and either $\eta = dx \wedge dy$ ($\eta$ has a unique double pole, i.e. $\div \eta = -2\{y=\infty\}$) or $\eta = dx\wedge\frac{dy}{y}$ ($\eta$ has simple poles along $\{y=0\}$ and $\{y=\infty\}$).  Since $f(x,y) = (\check f(x),f_2(x,y))$, and the base map $\check f$ preserves $dx$, it follows that the restriction $f_2(x_0,\cdot):\{x=x_0\}\to \{x = f(x_0)\}$ of $f$ to a general fiber must preserve either $dy$ or $dy/y$, respectively.  If fiber maps preserve $dy$, then they are affine, and $f$ is birational.  If fiber maps preserve $dy/y$, then we have $f_2(x_0,y) = \alpha(x_0)y^k$, where $R$ varies holomorphically with $x_0$.

When $g>1$, we again conjugate so that $S$ is a $\cp^1$-bundle over $B$ and $\div\eta = -2\sigma$.   We have $\check f^n = \id$ for some $n\ge 1$ so by trivializing $S\to B$ over any coordinate disk $U\subset B$, we can arrange that $\phi^{-1}(U) = U \times \cp^1$, $\eta = dx\wedge dy$ (in particular, the section $\sigma$ is sent to $\{y=\infty\}$) and $f^n|_U:(x,y)\mapsto (x,f_2(y))$.  Since $f$ preserves $\eta$, $f_2$ must preserve $dy$.  That is, $f_2$ is affine.  We conclude that $f^n$ and therefore also $f$ is birational.
\end{proof}

The maps in Theorem \ref{irrationalegs} need not be algebraically stable.  However, Theorem \ref{irrationalegs}, Theorem {\sbe} and \cite{DiFa} implies that if $f:S\to S$ is rational map preserving a zero-free two-form on a fibered surface $S\to B$ with $g=\mathrm{genus}\,B > 1$, then $f$ is automatically corrigible.  In fact the arguments of \cite[Theorem 0.1]{DiFa} extend to maps that are merely \emph{locally} birational, so maps in the irrational euclidean case of Theorem \ref{irrationalegs} are also corrigible.  The irrational smooth case is different, and we will treat it together with the rational smooth case in Corollary \ref{corrigiblesmooth}.  For now, we point out that when the induced map on the base $\check f = \id$, and the fiber maps are not affine, then recent work of DeMarco and Faber \cite{DeFa} gives a very different approach to algebraic stability.

\subsection{Surfaces $S$ with non-negative Kodaira dimension}
\label{kodimge0}

The results of this paper are more or less irrelevant when $\kod(S) \geq 0$.

For instance if $\kod(S) > 0$, the map $\phi_m:S\to \cp^N$ given by sections of $mK_S$ is non-trivial for $m$ large enough, and since $f$ induces a linear pullback $f^*:H^0(S,mK_S)\to H^0(S,mK_S)$, it follows that $f$ preserves fibers of $\phi_m$ and acts on the image $\phi_m(S)\subset\cp^N$ by the restriction of a linear transformation.  When $\kod(S) = 2$, $\phi_m$ is birational, and it follows that $\lambda_1 = \lambda_2= 1$.  When $\kod(S) = 1$, the image $\phi_m(S)$ is a curve, and it follows that 
$\lambda_1 = \lambda_2$.  So the hypothesis of Theorem {\zeradicate } is never satisfied.  Moreover, there are no zero-free meromorphic sections of $H^0(S,K_S)$ for any $m>0$, so Theorems {\sbe } and {\corrig } are also inapplicable to this case.

If $\kod(S) = 0$, then it was explained in \cite{DDG1} that up to birational conjugacy and finite cover one can assume that $S$ is a complex torus or a K3 surface.  In these cases $S$ admits a nowhere vanishing holomorphic two-form $\eta$ which must be preserved by $f$.  It follows that $f$ cannot contract curves and is therefore automatically algebraically stable.  Moreover, there are no non-trivial elaborations of $\eta$ (see Lemma \ref{orderdecreases} below).  So Theorems {\sbe } and {\corrig } are not relevant for this particular form $\eta$.  If $\tilde\eta$ is another two-form preserved by $f$, and $\tilde\eta$ is not a multiple of $\eta$, then $R:= \tilde\eta/\eta$ is a non-constant rational function satisfying $R\circ f = cR$ for some $c\in\C^*$.  It follows that $f$ preserves a fibration $S\to \cp^1$ and that the induced map on the base is linear.  In particular $\lambda_1(f) = \lambda_2(f)$ (see \cite[Lemma 4.1]{DDG1}).  So Theorem {\zeradicate } is not very helpful here.  Since $R$ is non-constant, $\tilde\eta$ cannot be zero-free so Theorems {\sbe } and {\corrig } do not apply with $\tilde \eta$ in place of $\eta$.

\subsection{Non-projective K\"ahler surfaces}
We assume that surfaces in this paper $S$ are projective largely for convenience of terminology (e.g. `rational' instead of `meromorphic' map, etc.).  A careful reader will note that Theorems {\zeradicate }, {\sbe }, and {\corrig } and their proofs work on \emph{all} compact K\"ahler surfaces.  However, non-projective K\"ahler surfaces always have non-negative Kodaira dimension, so while the theorems are valid, the discussion in the previous subsection shows that they are not very enlightening in the non-projective setting.

\section{Eliminating zeros of invariant two-forms}
\label{classification}
In this section we prove Theorem {\zeradicate}.  So to begin with, $f:S\to S$ is a rational map, and $\eta$ is a meromorphic two-form satisfying $f^*\eta = \delta\eta$ for some $\delta\in\C^*$.  As above, we take $\lambda_1$ and $\lambda_2$ to be the first dynamical and topological degrees of $f$, respectively.
Because $1\leq \lambda_2\leq \lambda_1^2$, Theorem {\zeradicate } holds automatically when $\lambda_1 = 1$.  So we may assume $\lambda_1 > 1$ in what follows.

Since we will modify the domain of $f$ in stages, it will be convenient to use subscripts to keep track of the rational surface(s) we are working on at any given moment.  Suppose $\pi_X:X\to S$ and $\pi_Y:Y\to S$ are birational.  We let $f_{XY}:X \to Y$ denote the lift of $f$, writing $f_X$ in place of $f_{XX}$ when $\pi_X = \pi_Y$.  We also set $\eta_X := \pi_X^*\eta$, $\eta_Y := \pi_Y^*\eta$.  In later sections, we will employ subscripting for points and curves in $X$, too.
Usually, $\pi_X$ and $\pi_Y$ will be modifications, and then we call $f_{XY}$ a `modification' of $f$.

\begin{lem}
\label{orderdecreases}
Let $\pi_X:X\to S$ and $\pi_Y:Y\to S$ be modifications of $S$.  For any irreducible curve $C\subset X$ such that
$\ord(C,\eta_X)\geq 0$, we have $\ord(C,\eta_X) \geq \ord(f_{XY}(C),\eta_Y)$ with equality if and only if $f_{XY}$ does not contract $C$ and $m(f_{XY},C) = 1$.
\end{lem}

\begin{proof}
The case where $f_{XY}(C)$ is a curve is an immediate consequence of Proposition \ref{jformula} and invariance of $\eta_X$.  Otherwise, we may choose a modification $\pi:Z\to Y$ so that $f_{XZ}$ does not contract $C$.  Proposition \ref{jformula} and invariance give again that
$$
\ord(C,\eta_X) \geq \ord(f_{XZ}(C),\eta_Z).
$$
However, $\pi$ contracts $f_{XZ}(C)$, so if $p = f_{XY}(C) = \pi(f_{XZ}(C))$, then
$$
\ord(f_{XZ}(C),\eta_Z) = \ord(f_{XZ}(C),\pi^*\div\eta_Y) + \ord(f_{XZ}(C),\crit(\pi)) \geq \ord(p,\eta_Y) + 1.
$$
\end{proof}

\begin{lem}
\label{intersection}
There exists a modification $\pi_X:X\to S$ such that the decomposition $\div\eta_X = D^+ - D^-$ into effective and anti-effective divisors satisfies $\ord(p,D^-) \leq 1$ for each $p\in\supp D^+$.  That is, $p$ lies in at most one irreducible component $C$ of $D^-$, and if $C$ exists, it is a simple pole for $\eta_X$, and $p$ is a regular point of $C$.
\end{lem}

\begin{proof}
First let $\pi_X$ be \emph{any} modification of $S$, let $p\in X$ be any point in $\supp D^+\cap \supp D^-$, and let $\sigma:\hat X\to X$ be the blowup of $X$ at $p$.  Let $E$ be the $-1$ curve contracted by $\sigma$.  Then the decomposition $\div\sigma^*\eta =  \hat D^+ - \hat D^-$ into effective and anti-effective divisors may be written
$$
 \hat D^+ - \hat D^- = \sigma^*\div\eta + E = (\sigma^*D^+ - kE) - (\sigma^*D^- -(k+1)E)
$$
for some $k\geq 0$, with $k=0$ if and only if $\ord(p, D^-)=1$.   Hence $\hat D^+\cdot \hat D^- \leq D^+\cdot D^-$ with equality if and only if $k=0$.

From this we infer that we may achieve the first conclusion by blowing up any point $p\in \supp D^+\cap \supp D^-$ such that $\ord(p,D^-) > 1$.  Since $D^+\cdot D^-$ is strictly decreased by this, it follows that after finitely many such blowups, $\ord(p,D^+) \leq 1$ for all $p\in \supp D^-$.
\end{proof}

Let $\divplus\subset \div X$ denote the finite dimensional subspace of divisors supported on the irreducible components of $D^+$.

\begin{lem}
\label{nonexpansion}
Let $\pi_X$ be as in the conclusion of Lemma \ref{intersection}.  The pullback operator $f_X^*:\div(X)\to\div(X)$ restricts to an operator $f_X^*:\divplus\to\divplus$.
No eigenvalue of $f_X^*|_{\divplus}$ has magnitude larger than $1$.
\end{lem}

\begin{proof}
For the first assertion, it suffices to show for each irreducible component $C$ of $\divplus$ and each irreducible component $C'$ of $f_X^*C$ that $\ord(C',\eta_X)>0$.  Lemma \ref{orderdecreases} and invariance of $\eta_X$ give us that $\ord(C',\eta_X) \geq \ord(f_X(C'),\eta_X)$.  If $f_X(C') = C$, then the right side of this inequality is positive.  If $f_X(C')$ is a point, then the inequality is strict, and Lemma \ref{intersection} tells us the right side is non-negative.  In either case $\ord(C',\eta_X)>0$.

Concerning the second assertion, let $\lambda$ be an eigenvalue for $f_X^*|_{\divplus}$ of largest magnitude.  Since $f_X^*$ preserves effective divisors, the Perron-Frobenius Theorem implies that $\lambda\geq 0$ and there exists a non-trivial effective divisor $D\in\divplus$ such that $f_X^* D = \lambda D$.  We must show $\lambda\leq 1$.  Let $C$ be an irreducible component of $D$ such that $\ord(C,\eta_X)$ is minimal and (given this) $\ord(C,D)$ is maximal.  Then by Lemma \ref{orderdecreases} we have one of three possibilities for $f_X(C)$.  If $f_X(C)\not\subset \supp D$, then $C \not\subset \supp f_X^*D$, and it must be that $\lambda = 0$.  If $f_X(C) = C'$ for some other irreducible component of $D$, then invariance and Proposition \ref{jformula} give $\ord(C,\eta_X) + 1 = m(f,C)(\ord(C',\eta_X)+1)$.  So $0<\ord(C,D)\leq \ord(C',D)$ implies $\ord(C,\eta_X) = \ord(C',\eta_X)$ and $m(f,C)=1$.  It follows that the weight of $C$ in $f^*D$ is the same as that of $C'$ in $D$, which is no larger than that of 
$C$ in
$D$.  So $\lambda\leq 1$.

The remaining possibility is that $f_X$ contracts $C$ to a point $p\in C'$ for some irreducible component $C'$ of $D^-$.   Let $\pi:Y\to X$ be a modification such that $f_{XY}$ such that $E=f_{XY}(C)$ is a curve.  Then
$$
\ord(C,\eta_X)  \geq \ord(E,\eta_Y) \geq \ord(p,\eta_X) + 1 \geq \ord(C',\eta_X) \geq \ord(C,\eta_X),
$$
where the first two inequalities come from Lemma \ref{orderdecreases} applied to $f_{XY}$ and $\pi$, the third follows from Lemma \ref{intersection} and the fourth from our choice of $C$.  It follows that all inequalities are actually equalities.  Note that equality in the first inequality implies $m(f_{XY},C) = 1$.  Equality in the third inequality and Lemma \ref{intersection} imply that $\ord(p,D^-) = 1$ and $\ord(p,D^+) = \ord(C',D^+)$; hence $p$ is a smooth point of $C'$ and no other irreducible component of $D^+$ contains $p$.  Equality in the fourth inequality and our choice of $C$ implies that $\ord(C,D) \geq \ord(C',D)$.  Therefore,
\begin{eqnarray*}
\lambda & = & \frac{\ord(C,f_X^*D)}{\ord(C,D)} \leq  \frac{\ord(C,f_X^*D)}{\ord(C',D)}= \ord(C,f_X^*C') = \ord(C,f_{XY}^*\pi^*C') \\
              & = & \ord(C,f_{XY}^*E)\ord(E,\pi^*C') = \ord(C,f_{XY}^*E) = m(f_{XY},C) = 1.
\end{eqnarray*}
\end{proof}

Lemma \ref{nonexpansion}, Proposition \ref{perp} and Theorem \ref{hodge} together imply

\begin{cor}
\label{dichotomy}
Let $\pi_X$ be as in the conclusion of Lemma \ref{intersection}, and suppose the first dynamical degree of $f$ satisfies $\lambda_1>1$.  Then the restriction of the intersection form to $\divplus$ is non-positive.  It is negative definite if and only if there is a modification $\pi:X\to \check X$ such that $-\div\eta_{\check X}$ is effective.  Otherwise $\lambda_1 = \lambda_2$, and there is a non-trivial effective divisor $D_*\in\divplus$ with the following additional properties.
\begin{itemize}
 \item $D_*\cdot D = 0$ for all $D\in\divplus$.
 \item $D_*$ is nef.
 \item $D_* \sim f^*D_* \sim \lambda_1^{-1}f_*D_*$, where $\sim$ denotes linear equivalence.
\end{itemize}
If $D_*'$ is another class with the same properties, then there are integers $k,k'>0$ such that $k'D'_*\sim k'D_*'$.
\end{cor}

\begin{proof}  Let $r_1\geq \lambda_1 > 1$ be the eigenvalue of largest magnitude for $f^*$, and choose a non-trivial nef class $\theta_*\in H^{1,1}(X)$ satisfying $f_*\theta_* = r_1\theta_*$.  Lemma \ref{nonexpansion} implies that for each $D\in\divplus$, we have $r_1^{-n}(f^*)^{n}D \to 0$ for all $D\in\divplus$.  From Proposition \ref{perp} we infer $D\cdot \theta_* = 0$.  So Theorem \ref{hodge} tells us that the intersection form on $\divplus$ is non-positive.

Suppose it is actually negative definite. If $D^+$ is non-trivial then by hypothesis
$$
D^+ \cdot K_X = (D^+)^2 - D^+\cdot D^- \leq (D^+)^2 < 0.
$$
Consequently some irreducible component $C$ of $D^+$ satisfies $C\cdot K_X < 0$.  Since $C^2<0$, too, we obtain that $C$ is a smooth rational curve with self-intersection $-1$, and there is a point blowup $\sigma:X\to\check X$ contracting it.  Now $\sigma_*D^+$ is the zero divisor for $\eta_{\check X}$ and $\sigma^*\div(\sigma_* D^+) \subset \divplus$.  Since $\sigma^*$ preserves intersections, it follows that the intersection form remains negative on $\div(\sigma_*D^+)$.  So we can repeat our argument until all components of $D^+$ are contracted.  I.e. we arrive at a modification $\pi:X\to\check X$ such that $\div \eta_{\check X} \leq 0$.  Conversely, if such a modification exists, it \emph{must} contract all components of $D^+$, so Grauert's criterion says that the intersection form is negative definite on $\divplus$.

Now suppose the intersection form is not definite on $\divplus$.  Then Theorem \ref{hodge} tells us that after replacing $\theta_*$ with a suitable positive multiple, there is an effective divisor $D_*\in\divplus$, unique up to linear equivalence, with Chern class equal to $\theta_*$.  Necessarily $D_*^2 = D_*\cdot D = 0$ for all $D\in\divplus$, and $f_*D_* \sim r_1D_*$.

Pullback preserves non-trivial nef classes, and Lemma \ref{nonexpansion} applies to any iterate of $f$, so for all $n\in\N$, we have that $(f^n)^*D_*$ is a non-trivial nef and effective divisor in $\divplus$.  It follows that $(f^n)^*D_* \sim s_nD_*$ for some $s_n > 0$.  Lemma \ref{nonexpansion} and integrality of Chern classes imply more specifically that $s_n=1$ for all $n$.

Applying Theorem \ref{pushpull} to $D_*$ and $f^n$ then gives
$$
(f^n)_* D_* \sim (f^n)_*(f^n)^* D_* = \lambda_2^n D_* + E_n^-(D_*),
$$
where $E_n^-(D_*)$ is an effective divisor.  We intersect both sides with $D_*$ and use $D_*^2 = 0$ to further obtain
$$
0 = (f^n)^*D_*\cdot D_* = D_*\cdot (f^n)_* D_* = E_n^-(D_*)\cdot D_*,
$$
so Theorem \ref{pushpull} tells us that $E_n^-(D_*) = 0$.  That is, $(f^n)_* D_* = \lambda_2^n D_*$ for all $n\in\N$.   Taking $n=1$, we see that $\lambda_1\leq r_1 = \lambda_2$.  Letting $n\to\infty$ and choosing a norm on $H^{1,1}(X)$ we also infer
$$
\lambda_1 \geq \lim_{n\to\infty} \norm{(f^n)_* D_*}^{1/n} = \lambda_2.
$$
Hence $\lambda_1 = r_1 = \lambda_2$.
\end{proof}

The proof of Corollary \ref{dichotomy} raises the possibility that some multiple of $D_*$ is a fiber of an $f$-invariant fibration of $X$, but we do not know whether this is actually true in general.  

\section{Elaboration}
\label{elaboration}

For the rest of this article, $\eta$ will be a zero-free two-form on a surface $S$, i.e. $-\div\eta$ is effective.  We will again have reason to pass from $S$ to various modifications $\pi_X:X\to S$, but now we will be mostly concerned with modifications for which $\eta_X:=\pi_X^*\eta$ is also zero-free.  Recall that $\pi_X$ is an \emph{elaboration} of $\eta$ if $\eta_X$ is zero-free.

\begin{defn}
An elaboration $\pi_X$ of $\eta$ is
\begin{itemize}
 \item \emph{strong} if every curve contracted by $\pi_X$ is a pole of $\eta_X$;
 \item \emph{peripheral} if no curve contracted by $\pi_X$ is a pole of $\eta_X$;
 \item a \emph{(further) elaboration of $\eta_Y$} if $\pi_X = \pi_Y\circ\pi$ factors through an elaboration $\pi_Y$ of $\eta$.
\end{itemize}
In the opposite direction, $\pi:X\to Y$ is an \emph{anti-elaboration} of $\eta_Y$ if all curves contracted by $\pi$ are zeros of $\eta_X$.
\end{defn}

If $\pi_X$ elaborates $\eta$, then we call a point $p_X\in X$ \emph{simple} for $\eta_X$ if $\ord(p_X,\eta_X) = -1$ and \emph{multiple} if $\ord(p_X,\eta_X)<-1$.  Note that we now use subscripting for points and curves, too.  The reason for this will become clear shortly.  For any modification $\pi:Y\to X$ and any irreducible $C_X\subset X$, we have $\ord(C_X,\eta_X) = \ord(C_Y,\eta_Y)$, where $C_Y\subset Y$ is the strict transform of $C_X$ by $\pi$.  Hence

\begin{prop}
\label{composition}
A composition of two modifications $\pi:Y\to X$, $\pi_X:X\to S$ elaborates $\eta$ if and only if $\pi_X$ elaborates $\eta$ and $\pi$ elaborates $\eta_X$.  The same is true for strong and peripheral elaborations.
\end{prop}

Other basic properties of elaborations proceed inductively from the case of a point blowup.

\begin{prop}
\label{blowup}
Suppose $\pi_X$ elaborates $\eta$, and $\sigma:\hat X \to X$ is the blowup of $X$ at $p_X\in X$.   Then $\sigma$ elaborates $\eta_X$ if and only if $p_X\in\supp\div\eta_X$.  The elaboration is peripheral if $p_X$ is simple for $\eta_X$ and strong if $p_X$ is multiple.
\end{prop}

\begin{proof}
Let $E_{\hat X}$ be the curve contracted by $\sigma$.  All conclusions proceed directly from $\crit(\pi) = E_{\hat X}$, i.e. $\ord(E_{\hat X},\eta_{\hat X}) = \ord(p_X,\eta_X) + 1$, .
\end{proof}

\begin{cor}
\label{blowupcor}
Let $\pi_X$ be an elaboration of $\eta$ and $\pi:Y\to X$ be a modification.  Suppose $C_Y\subset Y$ is an irreducible curve contracted by $\pi$.
\begin{itemize}
 \item If $\pi$ is an anti-elaboration, then $\pi(C_Y)\notin\supp\div\eta_X$.
 \item If $\pi$ is an elaboration, then $\pi(C_Y)\in \supp\div\eta_X$.
 \item If $\pi$ is a peripheral elaboration, then $\pi(C_Y)$ is a simple point for $\eta_X$.
 \item If $\pi$ is a strong elaboration, then $\pi(C_Y)$ is a multiple point for $\eta_X$.
\end{itemize}
\end{cor}

\begin{cor}
\label{decomps}
Let $\pi_X$ be an elaboration of $\eta$.
\begin{itemize}
 \item Any further modification $\pi:Y\to X$ decomposes uniquely $\pi=\pi_{elab}\circ \pi_{anti}$ into an elaboration and an anti-elaboration.  The zeros of $\eta_Y$ are disjoint from its poles.
 \item Any further elaboration $\pi:Y\to X$ decomposes $\pi=\pi_{str}\circ\pi_{per}$ uniquely into strong and peripheral elaborations.
 \item Any peripheral elaboration $\pi:Y\to X$ decomposes $\pi = \pi_1\circ\dots\circ\pi_k$ into peripheral elaborations whose critical divisors $\crit(\pi_j)$ have mutually disjoint and connected supports, each equal to a chain of smooth rational curves that meets $\supp\div \eta_Y$ in a single simple point $p_j$ whose image $\pi(p_j) = \pi(\crit(\pi_j))$ is simple for $\eta_X$.
\end{itemize}
 \end{cor}

\begin{proof}
All conclusions follow from Proposition \ref{blowup} and induction on the number of blowups comprising $\pi$.  For the first and third conclusions it suffices to note that we may switch the order of two blowups if they are centered over different points in $X$.

To prove the the second assertion, we write $\pi = \pi'\circ\sigma$, where $\pi' =\pi'_{anti}\circ\pi'_{per}:Z\to X$ is an elaboration, and $\sigma:Y\to Z$ is the blowup of a point $p_Z\in \supp\div\eta_Z$.  Let $E_Y$ be the curve contracted by $\sigma$.  If $E_Y$ is not a pole for $\eta_Y$, then we may take $\pi_{str} = \pi'_{str}$ and $\pi_{per} = \pi'_{per}\circ\sigma$.  Otherwise $\pi'_{per}\circ \sigma(E_Y)$ is a multiple point $p_X$ of $\eta_X$ and therefore not the image of any curve collapsed by $\pi'_{per}$.  So we may switch the order of $\pi'_{per}$ and $\sigma$ and take $\pi_{per} = \pi'_{per}$, $\pi_{str} = \pi'_{str}\circ \sigma$, where $\sigma$ is the blowup of $p_X$.
\end{proof}

\subsection{$\eta$-Primes and $\eta$-points}
It will be useful for us to coordinate information between different elaborations of $\eta$, so we adopt an idea (see e.g. \cite[\S V.35]{Man}) which was introduced into complex dynamics in the recent papers \cite{Ca1, BFJ} (see also \cite{Us}).  Our focus, however, will be on points and sets rather than cohomology classes.

If $\pi_X,\pi_Y$ both elaborate $\eta$, then there is always a common further elaboration $Z\to X,Y$ of both $\eta_X$ and $\eta_Y$.  We declare two curves $C_X\subset X$ and $C_Y\subset Y$ to be equivalent if they have the same strict transform in $Z$.   Equivalence does not depend on choice of $Z$.

\begin{defn}  Let $\pi_X$ be an elaboration of $\eta$.  An \emph{$\eta$-prime $C$ appearing in $X$} is the equivalence class of an irreducible curve $C_X\subset X$.  We call $C_X$ the \emph{incarnation} of $C$ in $X$.
\end{defn}

An $\eta$-prime $C$ may be regarded as a `divisorial' valuation $v_C$ on the function field $K(X)$.  That is, if $R\in K(X)$ is a rational function, then $v_C(R)$ is the order of vanishing of $R\circ\pi_X$ along some/any incarnation $C_X\subset X$.  This point of view is thoroughly developed in \cite{FaJo3} (see also \cite{Jon}).

Since $\eta$ is fixed, we will generally say `prime' instead of `$\eta$-prime'.  We set $\ord(C,\eta) := \ord(C_X,\eta_X)$, the right side being the same for any incarnation $C_X$ of $C$.  In particular $C$ is a \emph{(simple) pole} of $\eta$ if $C_X$ is a (simple) pole of $\eta_X$.  We let $\primes$ denote the set of all primes associated to $\eta$, $\poles\subset\primes$ the set of all poles, and $\simples\subset\poles$ the set of all simple poles.  In Proposition \ref{residue} below we will take advantage of the fact $\eta$ canonically induces a meromorphic one form (the Poincar\'e residue) on each of its simple poles.

If $C\in\primes$ does not appear in $X$, there exists a further elaboration $\pi:Y\to X$ incarnating $C$ and $\pi(C_Y)$ is a point in $X$ that is independent of $\pi$.  We refer to both the curve $C_Y$ and the point $\pi(C_Y)\in X$ as representatives of $C$.   Since elaboration makes curves less singular, there is a further elaboration $Z\to Y$ such that $C_Z$ is minimally singular; i.e. the number of singularities of $C_Z$ counted with multiplicity is minimal among all incarnations of $C$.  This is only an issue if $C$ appears in $S$, since otherwise all incarnations of $C$ are smooth.  A minimally singular incarnation of a pole of $\eta$ is always smooth.  Since finitely many poles of $\eta$ appear in $S$, any elaboration is dominated by a further elaboration in which all poles that appear do so smoothly.

Any two minimally singular incarnations $C_Z, C_{Z'}$ of $C$ are biholomorphic via the restriction of $\pi_{Z'}^{-1}\circ \pi_Z$ to $C_Z$, and this gives us an equivalence between points in $C_Z$ and $C_{Z'}$.  That is, we declare pairs $(p_Z,C_Z)$ and $(p_{Z'},C_{Z'})$ to be equivalent if  $p_Z\in C_Z$,  $p_{Z'}\in C_{Z'}$ and $p_{Z'} = (\pi_{Z'}^{-1}\circ \pi_Z)|_{C_Z}(p_Z)$.  This equivalence relation gives us a model-independent notion of \emph{points} in the support of the divisor of $\eta$

\begin{defn} An \emph{$\eta$-point} $p\in C\in\poles$ appearing in an elaboration $X$ of $\eta$ is the equivalence class of a pair $(p_X,C_X)$ where $C_X\subset X$ is a smooth incarnation of $C$ and $p_X\in C_X$.  We call $p_X$ the incarnation of $p$ in $X$.
\end{defn}

In the language of \cite{FaJo3} an $\eta$-point in $C$ is a tangent vector to the divisorial valuation associated to $C$. 

Let $\polepts := \bigcup_{C\in\poles} \{p\in C\}$ denote the set of all $\eta$-points contained in any pole of $\eta$.  As with primes, an $\eta$-point $p\in C$ is uniquely represented by some $p_X\in X$ even if $X$ does not incarnate $C$ smoothly.  If $C$ does not appear in $X$ then all $p\in C$ have the same representative in $X$.   We will say that an elaboration takes place over $p\in\polepts$ (resp, over $S\subset \polepts$) if $\pi$ decomposes as a sequence of blowups centered at points representing $p$ (resp, points representing elements of $S$).  Corollary \ref{blowupcor} says that any elaboration takes place over a finite subset of $\polepts$.

The order $\ord(p_X,\eta_X)$ of an incarnation of $p\in\polepts$ depends in general on the elaboration $\pi_X$.  However, whether $p_X$ is simple or multiple depends only on $p$, so we can single out the subset $\simplepolepts\subset\polepts$ of \emph{simple} $\eta$-points, all of which lie in simple poles of $\eta$.  A point $p\in\simplepolepts$ appears in $X$ if and only if its representative $p_X$ is simple for $\eta_X$.

As the reader will verify, simple/multiple points in a simple pole for $\eta$ may alternately be viewed as regular points/nodes for $\div\eta$:

\begin{prop}
If $C_X$ smoothly incarnates $C\in\simples$, then the representatives of the multiple points in $C$ are exactly those points where $C_X$ meets other components of $\div\eta_X$.  After a further strong elaboration such an intersection is a normal crossing between $C_X$ and one other pole of $\eta_X$.
\end{prop}

Now we turn to non-polar primes.  The main theme here is that any non-polar prime which does not appear in $S$ may be described in terms of its relationship with polar primes.

\begin{defn}
\label{exceptdef}
A prime $E\in\primes\setminus \poles$ is \emph{exceptional} if it does not appear in $S$.
\end{defn}

Corollary \ref{decomps} tells us that any  exceptional prime $E$ originates from a simple point in the following sense.  Let $E_X\subset X$ be an incarnation of $E$ and $\pi_X = \pi_{str}\circ\pi_{per}$ be the peripheral/strong decomposition of $\pi_X$.  Then $\pi_{per}(E_X) = \pi_{per}(p_X)$ for a unique simple point $p_X$ of $\eta_X$ in the incarnation $C_X$ of some $C\in\simples$.  The $\eta$-point $p\in C$ represented by $p_X$ is independent of the elaboration $X$.   We call $p$ the \emph{origin} of $E$ and write $p = \orig(E)$.

Conversely, let $X$ be a strong elaboration of $\eta$ incarnating some simple point $p\in C$.  Then any given exceptional prime $E$ originating from $p$ is incarnated by a further peripheral elaboration $\pi:Y\to X$ over $p$.  The \emph{depth} of $E$ will be the minimal number of blowups comprising an elaboration $\pi$ that incarnates $E$ in this fashion.  The depth $n\geq 1$ and the origin $p$ uniquely determine $E$, so we write $E = E_n(p)$.  If $E_n(p)$ appears in $X$, then so does $p$ and $E_m(p)$ for all $1\leq m < n$.  

\subsection{Rational surfaces}

Let us now discuss elaborations, primes and points associated to $\eta$ in each of the three cases in Proposition {\tricotomy}. For the rest of this section, we take $S=\cp^2$.

{\it Smooth case:} if $-\div\eta$ is a smooth cubic, then $\eta$ has no multiple points.  Hence all elaborations of $\eta$ are peripheral by Proposition \ref{blowup}, and $\poles = \simples$ consists of a single element $C$ which is a one dimensional complex torus.  All $\eta$-points are simple.

{\it Toric case:}  The form $\eta = \frac{dx\wedge dy}{xy}$ has simple poles along three lines, and each of the three points where two poles meet is a double point.  Blowing up one of these three points gives a strong elaboration $\sigma$ of $\eta$, and $\sigma^*\eta$ then has four simple poles and four double points.  Working inductively, one sees that for any strong elaboration $\pi_X$ of $\eta$, the divisor $\div\eta_X$ is a cycle of simple poles, each meeting the next in a double point.
This gives us the following (see the appendix on toric surfaces).

\begin{prop}
A modification $\pi_X:X\to \cp^2$ of $\cp^2$ is a strong elaboration of $\frac{dx\wedge dy}{xy}$ if and only if it is toric.  Consequently, $\poles$ is countable and coincides with $\simples$, and for each $C\in\poles$ the set $C\cap\simplepolepts$ of simple points in $C$ is isomorphic to $\C^*$.
\end{prop}

Indeed the appendix gives us a natural identification of $\poles$ with the set of rational rays in $\R^2$.

{\it Euclidean case: } The set $\poles$ (or rather something quite close to it) corresponding to $\eta=dx\wedge dy$ has been considered in detail by Favre and Jonsson \cite{FaJo1, FaJo2} from the point of view of valuations.  The idea here is that $\poles$ is naturally completed by the set $\val$ consisting of all normalized valuations that evaluate negatively on $\div\eta$, and the completion is a compact rooted metric tree.  The root is the prime represented in $\cp^2$ by the line at infinity. The set $\simples$ constitutes a dense set of endpoints in the tree.

Strong elaborations $\pi_X$ of $dx\wedge dy$ correspond (via the dual graph of $\crit\pi_X$ ) to finite subgraphs of $\val$ with vertices in $\poles$ that include the root.  Both $\simples$ and $\poles\setminus\simples$ are uncountable, and $\ord(C,\eta)$ is unbounded among $C\in\poles$.  All $C\in\poles$ are isomorphic to $\cp^1$, and if $C\in\simples$, then $C\cap\simplepolepts$ is isomorphic $\C$.

\section{Elaboration of rational maps}
\label{map elaboration}
Now let us consider a rational map $f: S\to S$ preserving the zero-free two-form $\eta$; i.e. $f^*\eta = \delta\eta$ for some constant $\delta\in\C^*$.  If $\pi_X:X\to S$, $\pi_Y:Y\to S$ are elaborations of $S$ then we say that $f_{XY}$ is an \emph{elaboration} of $f$.

\begin{prop}
\label{miscmapprop}
Suppose that $f_{XY}:X\to Y$ is an elaboration of $f$ and $C_X\subset X$ is an irreducible curve.
\begin{enumerate}
 \item If $f_{XY}$ does not contract $C_X$, then the prime incarnated by $f_{XY}(C_X)$, the local multiplicity of $f_{XY}$ along $C_X$, and the topological degree $\deg f_{XY}|_{C_X}$ of the restriction $f_{XY}|_{C_X}:C_X\to f(C_X)$ depend only on the prime incarnated by $C_X$.
 \item If $f_{XY}:X\to Y$ contracts $C_X$, then the image $f_{XY}(C_X)$ lies in $\supp\div\eta_Y$.  If in addition $C_X$ is polar, then $f_{XY}(C_X)$ is a multiple point for $\eta_Y$.
 \item There exists an elaboration $Z\to Y$ of the target such that $f_{XZ}$ does not contract $C_X$.
 \end{enumerate}
\end{prop}

\begin{proof}
The first assertion follows directly from the fact that, by definition, any two elaborations of $f$ are conjugate by a birational map that respects equivalence between irreducible curves. For the other assertions we assume that $f_{XY}$ contracts $C_X$.
If $\sigma:\hat Y\to Y$ is the blowup of $Y$ at the image $f_{XY}(C_X)$, then the order of $C_X$ in the critical set of $f_{X\hat Y}$ is strictly less than it is in the critical set of $f_{XY}$.  Since the order cannot be less than zero, it follows inductively that there is a modification $\pi: Y'\to Y$ such that $f_{XY'}$ does not contract $C_X$.  

Then from Proposition \ref{jformula} and $\ord(C_X,\eta_X) \leq 0$, we have that that $\ord(f_{XY'}(C_X),\eta_{Y'}) \leq 0$.  Factoring $\pi  =\pi_{elab}\circ \pi_{anti}$ into an elaboration $\pi_{elab}:Z\to Y$ and an anti-elaboration $\pi_{anti}:Y'\to Z$, we obtain that $\pi_{anti}$ does not contract $f_{XY'}(C_X)$.  Therefore neither does $f_{XZ} = \pi_{anti}\circ f_{XY'}$.  This proves the third assertion.  Proposition \ref{jformula} also tells us that $f_{XZ}(C_X)$ is a pole of $\eta_Z$ if and only if $C_X$ is a pole of $\eta_X$.  Therefore, the second assertion follows
from $f_{XY}(C_X) = \pi_{elab}(f_{XZ}(C_X))$ and Corollary \ref{blowupcor}.
\end{proof}

The first and third assertions in Proposition \ref{miscmapprop} imply that $f$ induces a well-defined map $\action:\primes \to \primes$.

\begin{defn}
Given $C\in\primes$, let $f_{XY}$ be an elaboration of $f$ such that $C$ appears in $X$ and $f_{XY}$ does not contract $C_X$.  Then we let $\action C\in \primes$ denote the prime incarnated by $f_{XY}(C_Y)$.  We also set
$m(f,C) := m(f_{XY},C_X)$ and $\deg f|_C:=\deg f_{XY}|_{C_X}$.
\end{defn}

Proposition \ref{jformula} has the following more or less immediate consequences.

\begin{cor}
\label{jcor} The following hold for all primes $C\in\primes$.
\begin{itemize}
\item $C$ is a (simple) pole of $\eta$ if and only if $\action C$ is.
\item if $f$ is branched about $C$ (i.e. $m(f,C)>1$), then either $C$ is simple, or $|\ord(\action C,\eta)| < |\ord(C,\eta)|$;
\item in particular $f$ is unbranched about non-polar and double polar $C$.
\end{itemize}
\end{cor}

As in \S\ref{classification}, we write $f_X$ instead of $f_{XX}$ to indicate an elaboration of $f$ with the same domain and range.

\begin{prop}
\label{strongcontract}
Let $f_X$ be an elaboration of $f$.  Then there exists a strong elaboration $\pi:Y\to X$ of $\eta_X$ such that if $f_Y$ contracts a non-polar curve $C_Y\subset Y$, then $f_Y(C_Y)$ is a simple point for $\eta_Y$.
\end{prop}

\begin{proof}
There are only finitely many non-polar curves $C_X$ contracted by $f_X$, and for each of these $\action C$ is an exceptional prime that does not appear in $X$.  Take $\pi:Y\to X$ to be any strong elaboration such that
$\orig(\action C)$ appears in $Y$ whenever $f_X$ contracts $C_X$.
\end{proof}

Unlike exceptional curves, points of indeterminacy cannot always be eliminated by elaboration.  The following result tells us that they can be separated somewhat from the divisor of $\eta$.  It implies that the restriction $\action:\poles \to \poles$ is surjective.

\begin{prop}
\label{indeterminacy}
Let $f_{XY}:X\to Y$ be an elaboration of $f$.  If $p_X\in X$ is not a multiple point for $\eta_X$, then every irreducible component of $f_{XY}(p_X)$ is non-polar.  Furthermore, there exist
\begin{itemize}
\item a strong elaboration $X'\to X$ such that $I(f_{X'Y})$ meets $\supp\div\eta_{X'}$ only at simple points; and
\item a further peripheral elaboration $X''\to X'$ such that $I(f_{X''Y})$ is disjoint from $\supp\div\eta_{X''}$.
\end{itemize}
\end{prop}

\begin{proof}
Let $\Gamma$ be the graph of $f_{XY}$.   Then projection $\pi:\Gamma \to X$ onto the domain is a modification of $X$ and therefore factors $\pi = \pi_{elab}\circ\pi_{anti} = \pi_{str}\circ\pi_{per}\circ \pi_{anti}$ into a strong elaboration $\pi_{str}:X'\to X$, a peripheral elaboration $\pi_{per}:X''\to X'$ and an antielaboration $\pi_{anti}:\Gamma \to X''$.  Note that $\Gamma$ is the graph of $f_{X'Y}$ and $f_{X''Y}$ as well as $f_{XY}$.

Fix $p_X\in I(f_{XY})$, and let $\tilde C_Y\in Y$ be any irreducible curve in $f_{XY}(p_X)$.  Then $\tilde C_Y = f_{\Gamma Y}(C_\Gamma)$ for some irreducible $C_\Gamma\subset \Gamma$ such that $p_X = \pi(C_\Gamma)$.  If $p_X$ is not multiple for $\eta_X$, then Proposition \ref{orderdecreases} gives us that $\ord(C_\Gamma,\eta_\Gamma) > \ord(p_X,\eta_X) \geq -1$, and then Proposition \ref{jformula} implies that $\ord(\tilde C_Y,\eta_Y) \geq 0$.  This proves the first assertion.

For the second assertion, fix $p_{X'}\in I(f_{X'Y})$.  Then as in the previous paragraph $f_{X'Y}(p_{X'})$ is a union of irreducible curves $f_{\Gamma Y}(C_\Gamma)$ where $\pi_{per}\circ \pi_{anti}$ contracts $C_\Gamma$ to $p_{X'}$.  Hence the assertion follows from the fact that the image of a curve contracted by $\pi_{anti}$ is a point of non-negative order for $\eta_{X''}$, and the image of a curve contracted by $\pi_{per}$ is a simple point for $\eta_{X'}$.  The argument for the third assertion is identical.
\end{proof}

\begin{cor}
\label{nononpolars}
Let $\pi_X$ be any elaboration of $\eta$.  Then there exists a strong elaboration $X'\to X$ of $\eta_X$ such that for all multiple points $p_{X'}$ of $\eta_Y$, $f_{X'}(p_{X'})$ contains no non-polar curves.  The same conclusion holds for any further strong elaboration $X''\to X'$ of $\eta_{X'}$.
\end{cor}

\begin{proof}
Take $X = Y$ in Proposition \ref{indeterminacy}, and let $\pi:X'\to X$ be the strong elaboration given in the first conclusion of the Proposition.  Replacing the target of $f_{X'X}$ with $X'$ might reintroduce multiple points to $I(f_{X'})$.  But since $X'\to X$ is a strong elaboration, the image of a multiple point by $f_{X'}$ will only contain poles of $\eta_{X'}$.  The same goes for further strong elaborations $X'' \to X'$.
\end{proof}

\begin{cor}
\label{invarianceofpoles}
We have $\action(\primes\setminus\poles)\subset \primes\setminus\poles$, $\action(\poles) = \poles$ and $\action(\simples) = \simples$.
\end{cor}

\begin{proof}
The inclusion `$\subset$' is from Corollary \ref{jcor} in all three cases.  Equality in the last two follows from Proposition \ref{indeterminacy}.
\end{proof}

\begin{rem}
\label{beyondelab}
If $X\to S$ is merely birational, i.e. not an elaboration or even a modification, then despite the fact that $\eta_X$ will typically have zeros as well as poles, the poles $C_X$ of $\eta_X$ still correspond naturally (i.e. via the same equivalence relation) to primes $C\in\poles$.   The main thing here is that the first conclusion in Corollary \ref{decomps}, together with the existence of a common modification $\Gamma\to X, S$, guarantees that there is an elaboration $X'\to S$ such that any pole of $\eta_X$ is equivalent to a pole of $\eta_{X'}$.

If, moreover, $Y\to S$ is another surface, then Proposition \ref{jformula} guarantees that $\action:\poles\to\poles$ still governs the behavior of $f_{XY}$ on poles of $\eta_X$, i.e. if $C_X\subset X$ is a pole of $\eta_X$, then $f_{XY}(C_X)$ is the point or irreducible curve in $\supp\eta_Y$ representing $\action C$. This will be useful to us in \S \ref{corrigibility} below.
\end{rem}

\begin{prop}
\label{simpletosimpleprop}
Let $f_{XY}$ be an elaboration of $f$ and $p_X\in X$ be any point not in $I(f_{XY})$.
\begin{itemize}
 \item If $p_X$ is multiple for $\eta_X$, then $f_{XY}(p_X)$ is multiple for $\eta_Y$.
 \item If $p_X$ is simple for $\eta_X$ and $f_{XY}$ does not contract the unique pole $C_X$ of $\eta_X$ that contains $p_X$, then $f_{XY}(p_X)$ is simple for $\eta_Y$.
\end{itemize}
\end{prop}

\begin{proof}
Suppose that $p_X$ is multiple for $\eta_X$.  Let $\sigma:\hat X\to X$ be the blowup at $p_X$.   Then the curve $E_{\hat X}$ contracted by $\sigma$ is a pole for $\eta_{\hat X}$.  Hence the second conclusion of Proposition \ref{miscmapprop} implies that $f_{\hat X Y}(E_{\hat X}) = f_{XY}(p_X)$ is multiple.

Now suppose that $p_X$ is simple, contained a unique simple pole $C_X$ of $\eta_X$, and that $f_{XY}$ does not contract $C_X$.  Let $\sigma:\hat Y\to Y$ be the blowup of $Y$ at $f_{XY}(p_X)$ and $E_{\hat Y} = \sigma^{-1}(f_{XY}(p_X))$ be the curve contracted by $\sigma$.  Then one of two things occurs.  The first possibility is that $p_X\in I(f_{X\hat Y})$ with $f_{X\hat Y}(p_X) = E_{\hat Y}$.  Then \ref{indeterminacy} tells us that $E_{\hat Y}$ is non-polar for $\eta_{\hat Y}$.  The second possibility is that there is a non-polar irreducible curve $C'_X\ni p_X$ contracted by $f_{XY}$ to $f_{XY}(p_X)$ and that $f_{X\hat Y}(C'_X) = E_{\hat Y}$.  Proposition \ref{jformula} then tells us again that $E_{\hat Y}$ is non-polar.  In either case, Corollary \ref{blowupcor} tells us that $f_{XY}(p_X)$ is simple.
\end{proof}

Let $p\in C\in\poles$ be an $\eta$-point.  Suppose $\pi_X,\pi_Y$ are elaborations smoothly incarnating $C$ and $\action C$, respectively.  Then even if $p_X$ is indeterminate for $f_{XY}$, the restriction of $f_{XY}$ to $C_X$ gives us a well-defined image $q_Y \in (\action C)_Y$ of $p_X$, and $q_Y$ incarnates a point $q\in\action C$ that is independent of the elaborations $\pi_X,\pi_Y$.  We set $\feta(p):=q$ and call $\feta:\polepts\to\polepts$ the \emph{restriction of $f$ to $\eta$}.  The restriction $f_\eta:C\to \action C$ of $f_\eta$ to any particular $C\in\poles$ is holomorphic and surjective.  Proposition \ref{simpletosimpleprop} immediately implies

\begin{cor}
\label{simpletosimple}
$\feta:\polepts\to\polepts$ preserves simple and multiple points.
\end{cor}

If $C$ is a simple pole of $\eta$, we can be more precise about the restriction $f_\eta:C\to \action C$.  Let $C_X$ be a smooth incarnation of $C$.  The form $\eta_X$ induces a meromorphic one form $\res_C\eta$ on $C_X$ via Poincar\'e residue:  in local coordinates $(x,y)$, chosen so that $C_X = \{x=0\}$,  we have $\eta = \frac{dx}{x}\wedge \tilde\eta$ for some one form $\tilde\eta$ whose restriction $\res_C\eta := \tilde\eta|_{C_X}$ is independent of both the elaboration and the choice of coordinate.  Moreover, $\res_C\eta$ is zero-free and the poles of $\res_C\eta$ are the same as the multiple points of $\eta$ along $C$.  It follows that either $C \cong \cp^1$ and $C$ contains one or two multiple points, or $C = C\cap \simplepolepts$ is a torus.  In all cases, we can choose a uniformizing parameter (i.e. a holomorphic universal covering) $z:\C\to C\cap\simplepolepts$ for the `simple part' of $C$ so that $\res_C\eta$ lifts to $dz$ on $\C$.

\begin{prop}
\label{residue}
For any $C\in\simples$, we have
$
f_\eta^* \res_{\action C}\eta = \frac{\delta}{m(f,C)}\res_C\eta.
$
So in terms of uniformizing parameters for the simple parts of $C$ and $f_\sharp C$, chosen to identify $\eta$ with $dz$, the map $\feta:C\cap\simplepolepts\to \action C\cap\simplepolepts$ is affine, given by $z\mapsto \frac{\delta}{m(f,C)}z + b$ for some $b\in \C$.
\end{prop}

\begin{proof}
Suppose $C$ appears in $X$ and $\action C$ in $Y$.   As above we choose local coordinates $(x,y)$ on both source and target that identify $C_X$ and $(\action C)_Y$ with $x=0$.  In such coordinates we have  $f(x,y) = (x^m g, h)$ where $g(0,y),h(0,y)\not\equiv 0$ and $m=m(f,C)$.  Hence
$$
\delta\frac{dx}{x}\wedge\tilde\eta = f_{XY}^*\eta_Y = m\frac{dx}{x}\wedge f^*\tilde\eta + \frac{dg}{g}\wedge f^*\tilde\eta.
$$
Multiplying through by $x$ and restricting to $x=0$ then gives
$$
\delta\res_C\eta = m f_\eta^*\res_{\action C} \eta.
$$
\end{proof}

\subsection{Action of $f$ on exceptional primes}

\begin{defn} An exceptional prime $E$ is \emph{tame} if $\action E$ is also exceptional and $\feta(\orig(E)) = \orig(\action E)$.  Otherwise $E$ is \emph{stray}.
\end{defn}

It will be important to work around stray exceptional primes as much as possible.  The next result allows us to incarnate all strays at once.

\begin{thm}
\label{straythm}
There are only finitely many stray exceptional primes.
\end{thm}

Recall from the discussion following Definition \ref{exceptdef} that if $p$ is a simple point of $\eta$, then $E_j(p)$ denotes the exceptional prime of depth $j$ originating at $p$.

\begin{lem}
\label{straylem1}
Let $p\in C\in\simples$ be a simple point and $m = m(f,C)$ be the local multiplicity of $f$ along $C$.  There exists $J\in\N$ and $\tau\in\Z$ such that for all $j\geq J$, we have $\action E_j(p) = E_{mj+\tau}(\feta(p))$.
\end{lem}

This fact is closely related to the last assertion of Theorem 3.1 in \cite{FaJo2}.

\begin{proof}
Let $X,Y$ be elaborations incarnating $p,q:=\feta(p)$ such that $p_X\notin I(f_{XY})$.   Let $C'_Y$ denote the simple pole of $\eta_Y$ containing $q_Y$.  We may assume (using e.g. Theorem 3.2 in \cite{cut1}) after further elaboration that  $f_{XY}$ is well-defined and monomial at $p_X$.  That is, there exist local coordinates about $p_X$ and $q_Y$ such that $C_X = C'_Y = \{x=0\}$ and in which $f_{XY}$ may be written $f_{XY}(x,y) = (x^a y^b, x^c y^d)$ for some $a,b,c,d\in\N$.   Our particular context allows us to be more precise.  I.e. $f_{XY}(C_X) = C'_Y$ implies $c=0$; and Proposition \ref{residue} tells us that $p_X$ is not critical for the restriction $f_{XY}:C_X \to C'_Y$, so $d=1$.  By further blowing up the target $Y$ at the point $(0,0)$ representing $p$, we may also arrange $b=0$.   The remaining exponent $a$ is then the local multiplicity $m$ of $f$ along $C$.  So without loss of generality $f_{XY}(x,y) = (x^m,y)$ in local coordinates.

Let $E_J(p)$ be the exceptional prime of maximal depth that appears in $X$ and $E_K(q)$ the exceptional prime of maximal depth that appears in $Y$.  Making $E_{J+\ell}(p)$, $\ell\geq 1$ appear in the source amounts to blowing up $p_X$, etc $\ell$ times.  This is a peripheral elaboration of $\eta_X$ which we denote by $\hat X$.  In local coordinates the incarnation of $E_{J+\ell}$ in $\hat X$ is given by $\{\hat y=0\}$ where $(x,y)=(\hat x{\hat y}^\ell,\hat y)$.  Then $f_{\hat X Y}(\hat x,\hat y)= (\hat x^m{\hat y}^{m\ell},\hat y)$, from which it follows that $\action E_{J+\ell}$ is the exceptional prime with depth $K + \ell m$.
\end{proof}

\begin{cor}
\label{straycor1}
For any $p\in\simples$, let $J\in\N$ be sufficiently large.  Then $E_j(p)$ is tame for $j\geq J$.  Let $\pi_X$ be an elaboration such that $E_J(p)$ is the deepest exceptional prime that originates from $p$ and appears in $X$; let $\pi_Y$ be an elaboration such that $\action E_J(p)$ is the deepest exceptional prime that originates from $\feta(p)$ and appears in $Y$.  Then $p_X\notin I(f_{XY})$ and there is no $j\leq J$ such that $f_{XY}$ contracts the incarnation of $E_j(p)$ to $\feta(p)_Y$.
\end{cor}

\begin{proof}
All but the very last conclusion follows if we take $J$ at least as large as in Lemma \ref{straylem1}.  Let $J'\geq J$ be larger than
the maximum depth of $\action E_j(p)$ over all $j< J$.  Replacing $J$ with $J'$ guarantees that the last conclusion holds, too.
\end{proof}

\proofof{Theorem \ref{straythm}}
Corollary \ref{straycor1} tells us that only finitely many stray exceptional primes originate from any given $\eta$-point.  We complete the proof of Theorem \ref{straythm} by showing that there are only finitely many $\eta$-points from which stray exceptional primes can originate.  So let $\mathcal{E}\subset\primes$ be a finite (but a priori arbitrarily large) set of stray exceptional primes.

Take $\pi_Y$, $\pi_X$ to be strong elaborations of $\eta$ such that for all $E\in\mathcal{E}$, we have that $\orig(E)$ appears in $X$ and that $\feta(\orig(E))$ appear in $Y$.  Using Proposition \ref{indeterminacy}, we can further assume that $I(f_{XY})$ contains no multiple points for $\eta_X$.  Under these circumstances, we claim first that $\orig(E)$ is represented by a point in $I(f_{XY})$ for all $E\in \mathcal{E}$, and second that $\# I(f_{XY})$ is bounded above by a number that depends only on $f$ and not on our choices of $X$ and $Y$.  Together, these claims imply that stray exceptional primes originate from only finitely many $\eta$-points, so it remains only to justify both of them.

To prove the first claim, suppose to the contrary that $E\in\mathcal{E}$ and $p:=\orig(E)$, but $p_X\notin I(f_{XY})$.  Then $f_{XY}(p_X) = q_Y$ where $q=\feta(p)$.  Hence if $X'$ is a peripheral elaboration of $X$ at $p$ that incarnates $E$, we have $f_{X'Y}(E_{X'}) = q_Y$.  This implies that $\action E$ is exceptional and that $\orig(\action E) = \feta(\orig(E)) = q$, contradicting the assumption that $E$ is stray.

For the second claim, let $p_X\in I(f_{XY})$ be any point and $C_Y\subset f_{XY}(p_X)$ be any irreducible component of the image.  Then Proposition \ref{indeterminacy} tells us that $C_Y$ is non-polar.  Since $\pi_Y$ contracts only poles of $\eta_Y$, it follows that $\pi_Y(C_Y)$ is a non-polar irreducible curve contained in $f(\pi_X(p_X))$.  That is $f_{XY}(I(f_{XY}))$ is a union of at most $N$ irreducible curves $C_Y$, where $N$ is the number of irreducible components of $f(I(f))$.  Each curve $C_Y$ is in turn the image of at most $\lambda_2$ points in $I(f_{XY})$, so we conclude that $\# I(f_{XY}) \leq \lambda_2N$.
\qed

\begin{cor}
\label{straycor2}
Let $X$ be any elaboration of $\eta$.  By further elaborating $\eta_X$ we may arrange that
\begin{enumerate}
\item all stray exceptional primes appear in $X$;
\item $f_X$ contracts only the incarnations of poles and exceptional primes;
\item the image of any indeterminate point $p_X\in \supp\div\eta_X$ likewise contains only incarnations of poles and exceptional primes; and
\item if the point $p_X$ in (3) is simple for $\eta_X$, then each irreducible component of $f_X(p_X)$ incarnates an exceptional prime originating at $\feta(p)$;
\end{enumerate}
These conditions continue to hold for any further elaboration $Y\to X$.
\end{cor}

\begin{proof}
Theorem \ref{straythm} guarantees an elaboration $X'$ in which all stray exceptional primes appear.  We can apply Proposition \ref{miscmapprop} to obtain an elaboration $X\to X'$ such that $f_{X'X}$ contracts no curves.  Applying the same elaboration to the source, we observe that $f_X$ can only contract incarnations of primes that do not appear in $X$.  Such primes are exceptional or polar, so (1) and (2) are satisfied by $X$ and clearly also for any further elaboration $Y\to X$.

We arrange condition (3) by using Proposition \ref{indeterminacy} to elaborate the target $X''\to X$ of $f_X$ so that $I(f_{XX''})\cap \supp\div\eta_Z = \emptyset$.  If $Y$ is any further elaboration of $X''$, and $p_Y\in \supp\div\eta_Y$, then any curve in the image $f_Y(p_Y)$ must be contracted by the combined elaboration $Y\to X'' \to X$.  The irreducible components of $f_Y(p_Y)$ therefore incarnate only poles and exceptional primes.  Redefining $X=X''$ then gives us condition (3).

Condition (4) is a consequence of conditions (1) and (3).  That is, each irreducible component $C_X$ of $f_X(p_X)$ incarnates an exceptional prime $C$ by condition (3) and Proposition \ref{indeterminacy}.  Hence $C = \action C'$ for some exceptional prime $C'$ originating at $p$ but not appearing in $X$.  From (1) we know that $C'$ is tame, so $C$ originates from $\feta(p)$.
\end{proof}

\subsection{Elaboration of maps on rational surfaces}

Here again we look at the particularly interesting case $S=\cp^2$.  
Recall that when $-\div\eta$ is smooth, there is a unique prime $C\in\poles$. The pole $C$ is simple and contains no multiple $\eta$-points.  As a curve $C$ is equivalent to its incarnation in $\cp^2$ and therefore a torus.

\begin{prop}
\label{mapofsmooth}
Suppose that $-\div\eta$ is smooth.  Then the unique prime $C\in \poles$ for $\eta$ satisfies $\action^{-1}(C) = C$.  Moreover, $\feta:C\to C$ is an unbranched $\left(\frac{|\delta|}{m(f,C)}\right)^2$-to-1 cover, and the topological degree of $f$ is $\lambda_2 = |\delta|^2/m(f,C)$.  Hence $|\delta|^2$ is an integer divisible by both $m(f,C)^2$ and $\lambda_2$.
\end{prop}

\begin{proof}
The first assertion is immediate from total invariance of $\poles$ by $\action$.  The second proceeds from Proposition \ref{residue}.  The formula for $\lambda_2$ follows from the second and the fact that in any elaboration $X$, if $p_X\in C_X$ is any simple point and $q_X\in X$ is a general point near $p_X$, then there are $m(f_X,C_X)$ preimages of $q_X$ near each preimage of $p_X$.
\end{proof}

Recall now that in the toric case, when $\eta = \frac{dx\wedge dy}{xy}$, each pole $C\in\primes$ is simple and corresponds to a rational ray in $\R^2$.  Moreover, $C$ is equivalent as a curve to $\cp^1$, with $C\cap \simplepolepts \cong \C^*$.  Hence $\res_C \eta = \frac{dz}{z}$.  In a little more detail, Theorem {\toricact} may be restated as follows.

\begin{thm}
\label{cover}
Suppose $\eta = \frac{dx\wedge dy}{xy}$.  Then for each $C\in\poles$, the restriction $\feta:C\to\action C$ is given by $\frac{|\delta|}{m(f,C)}$-fold cover branched at the multiple points of $C$.  Moreover, $\action:\poles\to\poles$ is given by a continuous, piecewise linear map $T_f:\R^2\to\R^2$ such that $T_f(\Z^2)\subset \Z^2$, and $T_f$ restricts to a $\frac{\lambda_2}{|\delta|}$-to-1 cover of $\R^2-\{\origin\}$ by itself.  Hence $\delta$ is an integer divisible by $m(f,C)$ and dividing $\lambda_2$.
\end{thm}

\begin{proof}
The first assertion proceeds from Proposition \ref{residue} and the fact that $C\cap\simplepolepts \cong \C^*$.  If $\pi_X$, $\pi_Y$ are elaborations smoothly incarnating $C,\action C\in\poles$, then $f_{XY}$ is locally $m(f,C)$ to $1$ in a neighborhood of a general point $p_X\in C_X$.  Hence for any small neighborhood $U_Y\supset C_Y$, the first assertion implies that there is a corresponding neighborhood $U_X\supset C_X$, such that $f_{XY}$ is $|\delta|$-to-1 from $U_X$ onto $U_Y$.  It follows that $\action:\poles\to\poles$ is $\frac{\lambda_2}{|\delta|}$-to-$1$.

To further understand $\action$, we rely on the discussion of toric surfaces in the appendix.  Given $C\in \poles$ corresponding to a pair $u = (a,b)\in \Z^2$, we fix strong elaborations $\pi_X$ and $\pi_Y$ incarnating $C$ and $\action C$, respectively.  Then for general $w = (\alpha,\beta) \in (\C^*)^2$ (i.e. chosen so $w\gamma^u(\C^*)$ does not land on $I(f_{XY})\cap C_X$), we have that $f_{XY}(w\gamma^u(\C^*))$ lands on $(\action C)_Y$.

In affine coordinates $(x,y)\in (\C^*)^2$,  the map $f_{XY} = (P_1/P_3,P_2/P_3)$, is given by polynomials $P_j(x,y)$.   And for any polynomial $P(x,y) =\sum_{i,j} c_{i,j} x^i y^j$, one computes that
$$
P(w \gamma^u(z)) \sim c(\alpha,\beta) z^{\nu(P)}
$$
where $\nu(P) = \min_{c_{i,j}\neq 0} ia + jb$ is a continuous piecewise linear function of $a,b$ that takes integer values when $a,b\in\Z$.  So for $z$ small
$$
f_{XY}(w\gamma^u(z)) \sim (c_1(\alpha,\beta) z^{\nu(P_1) - \nu(P_3)}, c_2(\alpha,\beta) z^{\nu(P_2)-\nu(P_3)}),
$$
Thus the action of $f$ on $\poles$ is induced by a continuous piecewise linear and integral map $T_f:\R^2\to \R^2$.  In particular, the finitely many rays along which $T_f$ fails to be linear all have rational slopes.  So from continuity of $T_f$, from density of rational rays in $\R^2$, and from the fact that $\action$ is $\frac{\lambda_2}{|\delta|}$-to-1, we infer that the restriction of $T_f$ to $\R^2-\{0\}$ is a cover of degree $\frac{\lambda_2}{|\delta|}$.
\end{proof}

If $f:(x,y)\to (x^ay^b,x^cy^d)$ is monomial, the induced action $\action:\poles\to\poles$ corresponds to the linear operator $T_f:\R^2\to\R^2$ with matrix $\left(\begin{matrix} a & b \\ c & d\end{matrix}\right)$.

If $f$ and $g$ both preserve $\frac{dx\wedge dy}{xy}$, then $T_{f\circ g} = T_f\circ T_g$.  For
$f$ birational this implies $T_{f^{-1}} = T_f^{-1}$.   Thus $T_f$ is a homeomorphism and $T_f(\Z^2) = \Z^2$.  I.e. $T_f$ is a `piecewise linear automorphism of $\Z^2$' in the terminology of \cite{Us}.
For the particular example $f:(x,y)\mapsto \left(y,\frac{1+y}{x}\right)$ from 2.1.2, one computes that
$T_f(i,j) = (j,\min\{-i,j-i\})$.   The birational maps described by \eqref{idaction} preserve the poles of $\frac{dx\wedge dy}{xy}$ component-wise and act biholomorphically near each of the three points where two poles meet.  Hence $T_f = \id$ for these $f$.

In all cases above, invertible or not, $T_f$ is a homeomorphism of $\R^2$ (though not an automorphism of $\Z^2$).  That is, as we pointed in in 2.1.2, $\delta$ not only divides $\lambda$, but is equal (up to sign). It would be interesting to know whether this is always true for maps preserving $\frac{dx\wedge dy}{xy}$.

For the sake of completeness, we recall that in the Euclidean case, when $\eta = dx\wedge dy$, we have $\simples\neq\poles$.  Moreover each $C\in \simples$ is equivalent to $\cp^1$ and contains exactly one multiple point.  Hence the simple points in $C$ may be identified with $\C$ in such a way that the restriction $\res_C\eta$ becomes $dz$.  Therefore Proposition \ref{residue} implies that the map $\feta:C\to\action C$ is always an isomorphism.

\section{Algebraic stability and Corrigibility}
\label{stability}

We now address the issue of corrigibility, defined in \S \ref{introduction}, for a rational map $f: S\to S$ preserving a zero-free two-form $\eta$.  The first step in this direction is to prove Theorem {\sbe}, which can be stated more precisely as follows.

\begin{thm}
\label{stablebyelaboration}
Suppose that $X$ is an elaboration of $\eta$ and that $\pi:Y\to X$ is a modification such that $f_Y$ is algebraically stable.  Let $\pi = \pi_{elab}\circ \pi_{anti}$ be the decomposition into elaboration $\pi_{elab}:\check Y\to X$ and anti-elaboration $\pi_{anti}:Y\to \check Y$ given by Corollary \ref{decomps}.  Then $f_{\check Y}$ is algebraically stable.
\end{thm}

Note that to prove this theorem and others in this section we will frequently (and without comment) use the geometric criterion for algebraic stability given in Proposition \ref{geometric criterion}.

\begin{proof}
Let $C_{\check Y}$ be an irreducible curve contracted by $f_{\check Y}$.  We claim that $f_Y$ contracts the strict transform $C_Y$ of $C_{\check Y}$ by $\pi_{anti}$.  Indeed $\pi_{anti}(f_Y(C_Y)) = f_{\check Y}(C_{\check Y}) $.  So if $f_Y$ does not not contract $C_Y$, then $\pi_{anti}$ contracts $f_Y(C_Y)$.  But then on the one hand, $\ord(f_Y(C_Y),\eta_Y)>0$ by definition of antielaboration; while on the other hand we have that $\ord(C_Y,\eta_Y) = \ord(C_{\check Y},\eta_Y)\leq 0$ implies $\ord(f_Y(C_Y),\eta_Y)\leq 0$ by Proposition \ref{jformula}.  This contradiction proves the claim.

Now suppose further that $f_{\check Y}^n(C_{\check Y})=p_{\check Y}\in I(f_{\check Y})$ for some minimal $n\in\N$, so that $f_{\check Y}$ fails to be algebraically stable.  Then because $f_Y$ is algebraically stable, the corresponding point $p_Y:=f_Y^n(C_Y)$ is not indeterminate for $f_Y$.  This implies that $p_Y$ lies in an irreducible curve $C'_Y$ contracted by $\pi_{anti}$ to $p_{\check Y}$.   But the first item in Corollary \ref{decomps} then yields $p_{\check Y}\notin\supp\div\eta_{\check Y}$, so that $f_{\check Y}^n$ contradicts Proposition \ref{miscmapprop}.  We conclude that $f_{\check Y}^n(C_{\check Y})\notin I(f_{\check{Y}})$ for any $n>0$, and $f_{\check Y}$ is algebraically stable.
\end{proof}

Recall from the introduction that $f$ is \emph{corrigible along $\eta$} if for any elaboration $X\to S$ of $\eta$, there is a further elaboration $Y\to X$ of $\eta_X$ such that no pole of $\eta_Y$ destabilizes $f_Y$.  If it exists, the further elaboration may be taken to be strong.  This follows from Proposition \ref{miscmapprop}, i.e. the image of a pole $C_Y$ of $\eta_Y$ is either another pole or a multiple point of $\eta_Y$.   Propostion \ref{miscmapprop} also tells us that checking whether a given pole destabilizes $\eta_Y$ can be accomplished in finitely many steps.  
For this reason it might seem that determining corrigibility along $\eta$ is much easier than determining full corrigibility for $f$.  Nevertheless, we will spend the rest of this section showing that corrigibility along $\eta$ implies corrigibility of $f$.  

\proofof{Theorem {\corrig}}
If $f$ is corrigible and $\pi_X$ is an elaboration of $\eta$, then by definition there is a modification $\pi:Y\to X$ such that $f_Y$ is algebraically stable.  By Theorem {\sbe} we may assume that $\pi$ is also an elaboration.  And since no irreducible curve in $Y$ destabilizes $f_Y$, this proves that $f$ is corrigible along $\eta$.

For the other direction, suppose instead that $f$ is corrigible along $\eta$.
Starting from any elaboration of $\eta$, we elaborate further so that the conclusions of Corollary \ref{straycor2} apply.
After a further strong elaboration, we may assume by Proposition \ref{strongcontract} that if $f_X:X\to X$ contracts a non-polar curve $C_X\subset X$, the image $f_X(C_X)$ is a simple point for $\eta_X$.  Invoking our hypothesis that $f$ is corrigible along $\eta$, we may finally assume that no pole of $f_X$ destabilizes $f_X$.   We will say that the elaboration $X$ is then \emph{prepared} for $f$.  Beginning with $X$ prepared, we will only need to employ peripheral elaborations to reach a surface on which $f$ becomes algebraically stable.

\begin{defn} A simple point $p$ belongs to the \emph{destabilizing locus} $\destab_X\subset\simplepolepts$ for $f_X$ if there exist $n,m\geq 0$ such that $\feta^n(p)$ is represented by a point in $I(f_X)$ and $p=\feta^m(q)$ for some $q$ represented by the image $f_X(C_X)$ of a curve contracted by $f_X$.
\end{defn}

\begin{lem}
\label{destable}
$\destab_X$ is a finite set, all of whose points appear in $X$.  Every destabilizing orbit for $f_X$ consists of incarnations of points in $\destab_X$.
\end{lem}

\begin{proof}
Suppose in order to reach a contradiction that some $p\in\destab_X$ does not appear in $X$, i.e. the representative $p_X$ is multiple for $\eta_X$.  Then we may choose $p = \feta^m(q)$ where $q_X$ is the image of an irreducible curve $C_X$ contracted by $f_X$ and $m\in\N$ is minimal in the sense that $\feta^j(q)$ appears in $X$ for all $0\leq j < m$.  Also there exists some minimal $n\in\N$ such that $q'_X \in I(f_X)$ where $q' = \feta^n(p)$.  Hence $q'_X = f_X^n(p_X)$ is multiple for $f_X$ by Proposition \ref{simpletosimpleprop}.  If $m=0$, then preparedness of $X$ implies that $C_X$ is polar.  If $m>0$, then $\feta^{m-1}(p)_X$ is simple for $\eta_X$, so the same proposition implies that $f_X(C'_X) = p_X$, where $C'_X$ is the unique simple pole containing $\feta^{m-1}(p)_X$.  Either way, we see that there is a pole of $\eta_X$ that destabilizes $f_X$, contradicting preparedness.  This proves that all points in $\destab_X$ appear in $X$.  In particular, distinct points in $\destab_X$ have distinct
representatives in $X$.

Concerning finiteness, we begin as in the previous paragraph with the fact that each $\eta$-point in $\destab_X$ lies in the forward orbit of some $q\in\destab_X$ such that $q_X$ is the image of an irreducible curve contracted by $f_X$.   Since $f_X$ contracts only finitely many irreducible curves, there are only finitely many such $q$, so it suffices to show that only finitely many points in the forward orbit of each $q$ lie in $\destab_X$.  If $q$ is preperiodic for $\feta$, i.e. $\feta^j(q) = \feta^k(q)$ for some $k>j>0$, then the entire forward orbit of $q$ is finite and there is nothing to prove.  So suppose $q$ is not preperiodic for $\feta$.  If $\feta^j(q)_X\notin I(f_X)$ for any $j\in\N$, then $\destab_X\cap \{\feta^j(q):j\in\N\} = \emptyset$.  Otherwise, the map $\destab_X \to X$ sending $\eta$-points to their representatives is injective, so that $\feta^j(q)_X \in I(f_X)$ for only finitely many $j$.   So if $J$ is the largest of these, we have that
$\destab_X\cap \{\feta^j(q):j\in\N\}  = \{q,\dots,\feta^J(q)\}$
is again finite.

Finally, let $C_X\subset X$ be a destabilizing curve for $f_X$ and $p_X := f_X(C_X), f_X(p_X),\dots,f_X^n(p_X)\in I(f_X)$ be the corresponding destabilizing orbit.  Preparedness tells us that $C_X$ cannot be a pole for $\eta_X$ and that $p_X$ must therefore be a simple point.  Hence $p,\dots,\feta^n(p)$ are simple points represented by $p_X,\dots, f_X^n(p_X)$.  By definition they belong to $\destab_X$.
\end{proof}

After the fact that it incorporates all destabilizing orbits, the main virtue of $\destab_X$ as we have defined it is that it can only decrease when we elaborate over it.

\begin{lem}
\label{decreases}
Let $p\in\destab_X$ be any point, and let $\sigma:Y\to X$ be the blowup of $X$ at $p_X$.  Then $Y$ is prepared for $f$, and $\destab_Y\subset \destab_X$.
\end{lem}

\begin{proof}
Let $E_Y = \sigma^{-1}(p_X)$.  By Lemma \ref{destable} $p_X$ is simple for $\eta_X$, so the class $E\subset \simples$ of $E_Y$ is a tame exceptional prime.  If $C_Y\neq E_Y$ is any irreducible curve contracted by $f_Y$, then $f_X$ contracts $C_X$.  So to see that $Y$ is prepared for $f$, it suffices to suppose that $f_Y$ contracts $E_Y$ and then show that $f_Y(E_Y)$ is a simple point for $\eta_Y$.  Tameness implies that $f_Y(E_Y)$ represents $\feta(p)$.  If $\feta(p)\in\destab_X$, then Lemma \ref{destable} tells us $\feta(p)$ appears in $X$.  If not then $p_X \in I(f_X)$.  Corollary \ref{straycor2} then tells us that $f_X(p_X)$ contains the incarnation $E'_X$ of an exceptional prime originating from $\feta(p)$.  As $E'$ appears in $X$, we again find that $\feta(p)$ appears in $X$.  In either case $\feta(p)$ appears in $Y$, too.  I.e. $f_Y(E_Y)$ is simple for $\eta_Y$.  So $Y$ is prepared for $f$.

The set $\destab_Y$ consists of orbit segments $q,\dots, q' := \feta^n(q)\in\simplepolepts$ such that $f_Y$ contracts some curve $C_Y$ to $q_Y$, and $q'_Y\in I(f_Y)$.  If $E_Y\neq C_Y$, then $f_X$ contracts $C_X$, and if $E_Y = C_Y$ then tameness implies $q = \feta(p)$.  So in either case, there exists \emph{some} curve $C'_X$ contracted by $f_X$ to a point $p'_X$ and $q$ is in the forward orbit of $p'$.  Similarly, if $E_Y\neq f_Y(q'_Y)$, then $q_X\in I(f_X)$.  And if $E_Y = f_Y(q'_Y)$, then $\feta(q') = p \in\destab_X$. So some point in the forward orbit of $q'$ appears as a point in $I(f_X)$.  It follows in all cases that the orbit segment $q,\dots, q'\in \destab_Y$ is contained in a similar orbit segment lying in $\destab_X$.
\end{proof}

The definition of $\destab_X$ implies that an $\feta$-periodic point belongs to $\destab_X$ if and only if all points in the cycle containing $p$ belong to $\destab_X$.

\begin{lem}
\label{nonperiodic}
After further peripheral elaboration of $\eta_X$, we may (additionally) arrange that all points in $\destab_X$ are periodic for $\feta$.
\end{lem}

\begin{proof}
Suppose $\destab_X$ contains nonperiodic points.  Because $\destab_X$ is finite, we can choose $p\in\destab_X$ such that no $\feta$-preimage of $p$ lies in $\destab_X$.   Necessarily, $p_X$ is the image of a curve contracted by $f_X$.

Let $\sigma:Y\to X$ be the blowup of $X$  at $p_X$.  Then Lemma \ref{decreases} guarantees that $\feta^{-1}(p)\cap \destab_Y = \emptyset$.  So $p\in\destab_Y$ means again that $p_Y$ is the image of a curve contracted by $f_Y$.  So we repeat the process, continuing to blow up incarnations of $p$.  By Proposition \ref{blowup}, finitely many such blowups lead to a peripheral elaboration $Y\to X$ such that $p_Y$ is not the image of any curve contracted by $f_Y$.  So in finitely many steps we reach an elaboration $Y\to X$ such that $\destab_Y\subset \destab_X\setminus\{p\}$.

Since $\destab_X$ is finite, we can repeat the above as often as necessary to eliminate all nonperiodic $p\in\destab_X$.
\end{proof}

The proof of Theorem {\corrig} is now completed by

\begin{lem}
\label{periodic}
If $\destab_X$ consists only of periodic points and $p\in\destab_X$ is any point, then there is a peripheral elaboration $Y\to X$ over the points in the cycle generated by $p$ such that $p\notin\destab_Y$.
\end{lem}

\begin{proof}
Let $n>0$ be the minimal period of $p$, and let $p^j := \feta^j(p)$ denote the points in the cycle generated by $p$.  By further elaboration of $\eta_X$ over this cycle, we first arrange the following.
\begin{itemize}
\item If for any $j$ we find that $f_X$ contracts an irreducible curve $C_X\subset X$ to $p^j_X$, then $C_X$ is a tame exceptional prime originated from $p^{j-1}$.
\item For $0\leq j \leq n-1$, let $E_j$ be the deepest prime that originates from $p^j$ and appears in $X$.  Then the depth of $E_j$ is large enough so that Corollary~\ref{straycor1} applies at $p^j_X$.
\item For $1\leq j\leq n-1$, we have $\action E_{j-1} = E_j$.
\end{itemize}
We then set $E_n := \action E_{n-1}$, and let $Y$ be the elaboration of $\eta$ that differs from $X$ in exactly the following way: the deepest exceptional prime that originates at $p$ and appears in $Y$ is $E_n$ rather than $E_0$.

Under these circumstances Corollary \ref{straycor1} tells us that $p^j_X$, $0\leq j\leq n-1$, is not indeterminate for $f_{XY}$ and that $p^j_Y$ is not the image of a curve contracted by $f_{XY}$.  The exceptional primes $E_0$ and $E_n$ both originate at $p=p_0=p_n$.  Suppose $E_0$ is the deeper of the two.  It follows that $X$ is obtained from $Y$ by further elaboration $\pi:X\to Y$ centered at $p$.  Thus $f_X = \pi^{-1} \circ f_{XY}$, and we infer that no point $p^j_X$ (including $j=0$) is the image of a curve contracted by $f_X$.
Similarly, if $E_n$ is deeper than $E_0$, then $f_X = \pi\circ f_{XY}$ for an elaboration $\pi:Y\to X$.  So
in this case no point $p_X^j$ (including $j=n-1$) is indeterminate for $f_X$.

Either way, we see that the cycle $p,\dots,\feta^n(p)$ is eliminated from $\destab_X$.
\end{proof}

\section{Corrigibility in the smooth and (especially) toric cases}
\label{corrigibility}

The next result, an easy consequence of Theorem {\corrig}, immediately gives corrigibility for any rational map that preserves a two-form equivalent to one of those given in the smooth cases of Proposition {\tricotomy } and Theorem \ref{irrationalegs}.  Note that these include many non-invertible rational maps.

\begin{cor}
\label{corrigiblesmooth}
If $\eta$ is a meromorphic two-form on a surface $S$ with $-\div\eta$ equal to a reduced smooth curve $C_S$, then any rational map $f:S\to S$ that preserves $\eta$ is corrigible.
\end{cor}

\begin{proof}
Let $\pi_X$ be any elaboration of $\eta$.  By Proposition \ref{blowup}, $-\div\eta_X = C_X$ is just the strict transform of $C_S$.  That is, $\primes = \poles$ is a finite set with each element incarnated by an irreducible component of $C$.   Corollary \ref{jcor} therefore implies that $f$ permutes the components of $C$ without contraction.  
So even without further elaboration, we see that $f$ is corrigible along $\eta$.  Theorem {\corrig} then gives that $f$ is corrigible.
\end{proof}

The rest of this section is devoted to establishing Theorem {\toriccor}, which is an analogue for Corollary \ref{corrigiblesmooth} in the toric case.  So let $\eta=\frac{dx\wedge dy}{xy}$ and $f:\cp^2\to\cp^2$ be a rational map preserving $\eta$.  The starting point for this discussion is Theorem {\toricact}, which equates $\action:\poles\to\poles$ with a piecewise linear and integral map $T_f:\R^2\to\R^2$.   We assume here that $T_f$ is a homeomorphism.  By regarding $T_f$ as a map on one dimensional rays, we obtain a circle homeomorphism $S^1\to S^1$, which we persist in calling $T_f$.

Let $F$ be a lift of $T_f:S^1\to S^1$ to $\R$. Then the \emph{rotation number} (see \cite[Chapter 11]{KaHa} for more details) of $T_f$, denoted by $\rho(T_f)$, is the quantity
\[
\rho(T_f):=\lim_{n\to\infty}\frac{F^n(x)}{n}.
\]
The limit always exists, and is independent (mod $\Z$) of the choices of lift $F$ and basepoint $x\in\R$.  One has $\rho(T_f) \in \Q$ if and only if $T_f$ has periodic rays. If $T_f$ is orientation reversing, then $\rho(T_f)=0$, $T_f$ has exactly two fixed rays, and all other periodic rays have minimal period two. If $T_f$ is orientation preserving and $\rho(T_f)=\frac{m}{n}$ with $m,n\in\Z$, $n>0$, and $\gcd(m,n)=1$, then all periodic points of $T_f$ have minimal period $n$.

\begin{lem}
\label{denjoy}
Suppose $\rho(T_f)\notin \Q$.  Then there is a homeomorphism $S^1\to S^1$ conjugating the action of
$T_f$ on rays to rotation by $\rho(T_f)$
\end{lem}

\begin{proof}
This is an instance of the classical Denjoy Theorem for circle homeomorphisms with irrational rotation number (see \cite[\S 12.1]{KaHa}), slightly generalized to accommodate for lower regularity of $T_f$ along the finitely many rays in $\R^2$ where linearity fails.  The version given in \cite[Theorem 2.5]{ADM} suffices.
\end{proof}

This discussion has the following fairly immediate consequence.  

\begin{cor}
\label{plaz2}
Let $T$ be a piecewise linear automorphism of $\Z^2$, i.e. a piecewise linear homeomorphism $T:\R^2\to\R^2$ such that $T(\Z^2)=\Z^2$. Then $T$ induces a homeomorphism $S^1\to S^1$ with rational rotation number.
\end{cor}

\begin{proof}
Suppose $T$ has irrational rotation number.  Then Lemma \ref{denjoy} implies that the images $T^n(\tau)$ of any ray $\tau\subset\R^2$ are dense in $\R^2$.

As we noted above, Usnich showed that $T = T_f$ for some plane birational map preserving $\eta = \frac{dx\wedge dy}{xy}$.  By \cite{DiFa}, $f$ is corrigible and Theorem~\ref{stablebyelaboration} therefore gives us an elaboration $\pi_X$ of $\eta$ such that $f_X$ is algebraically stable.  On the other hand, let $C_X\subset X$ be any irreducible component of $\div\eta_X$ and $\tau\subset\R^2$ be the corresponding rational ray.  Then for some (in fact any) two dimensional cone $\sigma$ in the fan for $X$ and some $n>0$, we have $n,m>0$ such that $T_f^n(\tau),T^{-m}(\sigma) \in \sigma$.  Thus $f_X^n(C_X) = p_X$ and $C_X\subset f_X^m(p_X)$, where $p_X$ is the double point of $\eta_X$ corresponding to $\sigma$.  This contradicts algebraic stability of $f_X$ and implies that $\rho(T)$ is rational after all.
\end{proof}

Corollary \ref{plaz2} was originally proven by very different methods in \cite{GhSe} and later given simpler proofs by Liousse \cite{Li} and Calegari \cite{Cal}.  The proof we give here can be recast in purely combinatorial terms, without any reference to algebraic geometry.  The resulting argument is reminiscent of the train tracks proof given by Calegari.

\begin{eg}
If $f$ is monomial, then $T_f:\R^2\to\R^2$ is linear.  Hence $\rho(T_f)$ is rational if and only if some iterate of $T_f$ has a real eigenvalue.  Since the matrix for $T_f$ has integer entries, it suffices to consider $T_f^6$. 
\end{eg}

\begin{eg}
The following examples are from \cite{Fa2,JW}. These are examples of monomial maps which become corrigible along $\eta$ only when passing to an iterate.  Hence passing to an iterate in the first two conclusions of Theorem {\toriccor} is necessary.
\begin{itemize}
\item \cite[Examples 3.12]{JW} For $A=\left(\begin{matrix} -1 & 3 \\ 3 & 2\end{matrix}\right)$, the corresponding
monomial map $f_A(x,y)=(x^{-1}y^3, x^3 y^2)$ cannot be made algebraically stable. In this case, $\det(A)=-11<0$ and
$f_A^2$ can be stabilized by toric blowups.
\item \cite[Examples 3.14]{JW} For $A=\left(\begin{matrix} -1 & -1 \\ 3 & -1\end{matrix}\right)$, the corresponding
monomial map $f_A(x,y)=(x^{-1}y^{-1}, x^3 y^{-1})$ cannot be made algebraically stable. In this case, $A^3=8\cdot\Id$ and
$f_A^3$ is in fact a morphism on any toric variety.
\end{itemize}
\end{eg}

\begin{eg}
Let $f_1:\cp^2\to\cp^2$ be a birational map in the family given by \eqref{idaction} in \S\ref{section_example_toric}, $f_2:\cp^2\to\cp^2$ be a monomial map, and $f:= f_1\circ f_2$.  Then as we observed earlier $T_{f_1} = \id$ so that $T_f = T_{f_1}\circ T_{f_2} = T_{f_2}$.  Hence by Theorem {\toriccor}, there exists an iterate $f^n$ of $f$ that is birationally conjugate to an algebraically stable map if and only if some power of the eigenvalues of $T_{f_2}$ is real.  This gives in particular many new examples of rational maps $f:\cp^2\to\cp^2$ that are not stabilizable by birational conjugacy.  We intend to consider the dynamics of some of these in detail in future article.
\end{eg}

Before proving Theorem {\toriccor}, let us point out one further consequence.

\begin{cor}
If $f:\cp^2\to\cp^2$ is a rational map preserving $\frac{dx\wedge dy}{xy}$ and $T_f$ is injective with
$\rho(T_f) \in \Q$, then the first dynamical degree $\lambda_1(f)$ is an algebraic integer.
\end{cor}

\begin{proof}
The hypothesis implies that $f^n$ is birationally conjugate to an algebraically stable map $f_X^n: X \to X$.  Hence
$\lambda_1(f^n) = \lambda_1(f_X^n)=\lambda_1(f)^n$ is the largest eigenvalue of the linear operator $(f_X^n)^*$ acting on $\pic(X) \cong H^2(X,\Z)$, thus an algebraic integer. As a consequence, $\lambda_1(f)$ is also an algebraic integer.
\end{proof}

\subsection{Proof of Theorem {\toriccor}}

From here we rely heavily on the notation and discussion in the appendix on toric surfaces.  Throughout this section $X$ will denote a smooth compact toric surface with fan $\fan$; $f_X:X\to X$ will denote the rational map obtained by conjugating $f:\cp^2\to\cp^2$ by the canonical birational map $\psi:X \to \cp^2$ described in the appendix; and $\eta_X = \psi^*\eta$ will be the meromorphic extension of $\frac{dx\wedge dy}{xy}$ to $X$.
Note that $\psi$ is not necessarily a modification of $\cp^2$ here.

\begin{lem} Let $\tau,\tau'$ be rays $\sigma,\sigma'$ be two dimensional cones in $\Delta$.  Let $C_X, C'_X$ and $p_X,p_X'$ be the associated irreducible components and double points for $\div\eta_X$.
\begin{enumerate}
\item $T_f(\tau) = \tau'$ if and only if $f_X(C_X) = C'_X$.
\item $T_f(\tau) \subset \sigma$ if and only if $f_X(C_X) = p_X$.
\item $T_f(\sigma) \subset \sigma'$ if and only if $T_f(p_X) \cap \supp\div\eta_X = p'_X$.
\item $\tau\subset T_f(\sigma)$ if and only if $C_X \in T_f(p_X)$.  In particular, $p_X\in I(f_X)$.
\end{enumerate}
\end{lem}

Note that there is subtle point in item (3) that does not arise when $f$ is a monomial map.  Namely, it can happen that $p_X$ is indeterminate for $f_X$ when $T_f(\sigma) \subset \sigma'$.  The lemma says in this case that no irreducible component of $f_X(p_X)$ is a pole of $\eta_X$.  However, if $X$ is chosen to satisfy the conclusion of Corollary \ref{nononpolars}, the fact that $p_X$ is multiple point for $\eta_X$ allows implies that $p_X\notin I(f_X)$ after all.

\begin{proof}
These are all more or less immediate, given the landing map correspondence between polar primes and rational rays.  For instance, the statements in the first conclusion both boil down to the statement that for any $u\in\tau$, we have $\lim_{z\to 0} f_X(w\gamma^u(z)) = \lim_{z\to 0} h(w)\gamma^{T_f(u)}$ for some non-constant holomorphic map $h:(\C^*)^2 \to (\C^*)^2$.
\end{proof}

\begin{cor}
An irreducible component $C_X$ of $\div\eta_X$ destabilizes $f_X$ if the associated ray $\tau\in\Delta$ satisfies $T_f(\tau)\subset \sigma$ for some two dimensional cone $\sigma\in\Delta$ and $T_f^k(\sigma)\supset\tau'$ for some $k>0$ and ray $\tau'\in\Delta$.  If $X$ satisfies the conclusion of Corollary \ref{nononpolars}, then the converse holds.
\end{cor}

Now we proceed to the proof of Theorem {\toriccor}.

\subsection{Case (1): $\rho(T_f) \notin \Q$}

Let $Y$ be any (smooth) rational surface and $\psi:Y\to \cp^2$ a birational map (not necessarily a modification).  Then the form $\eta_Y := \psi^*\eta$ must have at least one (irreducible) pole $C_Y\subset Y$, and $C_Y$ incarnates a prime $C\in\poles$ (see Remark \ref{beyondelab}).  Since $\rho(T_f)$ is rational, the $\action$-orbit of $C$ is infinite, and since only finitely many polar primes appear in $Y$, it follows that some iterate $f^k_Y$ contracts $C_Y$ to a point $p_Y\in Y$.  Let $\tau$ in $\R^2$ be the rational ray corresponding to $\action^k C_Y$.  Then there is an open cone $\sigma$ containing $C_Y$ such that all rational rays in $\sigma$ correspond to polar primes that do not appear in $Y$.  By Lemma \ref{denjoy}, there exists $\ell>k$ such that $\action^\ell(\sigma)\subset\tau$.  That is, $C_Y \subset f^\ell(p_Y)$ and in particular $p_Y\in I(f_Y)$.  Hence $f_Y$ fails the geometric criterion for algebraically stability given in Proposition \ref{geometric criterion}.  Since $\rho(T_{
f^n}) = \rho(T_f^n)\notin \Q$ for any $n>0$, the same argument shows that no iterate of $f_Y$ is algebraically stable.

\subsection{Case (2): $\rho(T_f) \in\Q$}

By Theorem~{\corrig} and Theorem~\ref{cover}, it is sufficient to show the result for toric surfaces.  Replacing $f$ by $f^n$, we may assume that $T_f$ is orientation preserving and $\rho(T_f) = 1$.

Hence $T_f$ has periodic rays and all of them are in fact fixed.  The set $\fix(T_f)$ of rays fixed by $T_f$ has finitely many connected components $K\subset S^1$.  A point component $K$ is given by a ray in a one-dimensional eigenspace for some linear `piece' of $T_f$.  Non-point (i.e. arc) components $K\subset S^1$ consist of all rays in a closed two dimensional cone on which $T_f:\R^2\to\R^2$ is a a multiple of the identity operator.  Either way, $K$ is:
\begin{itemize}
\item[1.] \emph{attracting} if there is an open arc $U\supset K$ such that $K = \cap_{n\geq 0} T_f^n(U)$.
\item[2.] \emph{repelling} if there is an open arc $U\supset K$ such that $K= \cap_{n\geq 0} T_f^{-n}(U)$.
\item[3.] \emph{semi-stable}: if there is an open arc $U\supset K$ such that one half of $U-K$ is attracted by $K$ and the other is repelled.
\end{itemize}

\begin{lem}
Isolated semi-stable fixed points and endpoints of fixed arcs are rational.
\end{lem}

\begin{proof}
Suppose first that $\tilde u$ is an isolated fixed point, equal to the projection of an eigenvector $u$ for the linear operator $T:\R^2\to\R^2$ defining $T_f$ at $u$.
Semistability implies that either $T_f$ fails to be linear at $u$ (i.e. $T$ defines $T_f$ only to one side of $u$), or the Jordan form of $T$ about $u$ is nontrivial, equal to $\left(\begin{matrix} \lambda & 1 \\ 0 & \lambda\end{matrix}\right)$. In either case, integrality of $T_f$ implies that $u$ has rational slope.

Similarly, the end-points of any fixed arc correspond to rays in $\R^2$ along which $T_f$ fails to be linear.  So once again the endpoints are rational.
\end{proof}

Now let $X\to \cp^2$ be any toric modification (i.e. strong elaboration of $\frac{dx\wedge dy}{xy}$).  From here the procedure is quite similar to the one used in \cite{Fa2,JW} to stabilize monomial maps.  By adding rays and subdividing cones in the fan $\Delta$ for $X$, we may assume the following.
\begin{itemize}
 \item $X$ is smooth, i.e. all two dimensional cones $\sigma\in\Delta$ are regular.
 \item $\Delta$ includes all rational rays that are not interior points of $\fix(T_f)$ (in particular all isolated semi-stable fixed points, and the endpoints of any arc in $\fix(T_f)$).
 \item The conclusion of Corollary \ref{nononpolars} applies; i.e. the image $f_X(p_X)$ of a multiple point $p_X$ contains no non-polar curves.
 \item A two-dimensional cone $\sigma\in\Delta$ contains at most one isolated attracting ray in $\fix(T_f)$, and if this ray is present, $T_f$ contracts $\sigma$, i.e. $T_f(\sigma)\subset \sigma$.
 \item Similarly, $T_f$ expands any two dimensional cone $\sigma\in\Delta$ that contains an isolated repelling ray.
 \end{itemize}
All these conditions persist under further strong elaboration of $\eta_X$.

Note also that for each ray $\tau\in\Delta$ not fixed by $T_f$, the forward orbit $T_f^n$ is `strictly monotone' (e.g. successive images always move clockwise around $\origin$) and eventually falls into a two dimensional $\sigma\in\Delta$ satisfying $T_f(\sigma) \subset \sigma$.  We refine $\Delta$ by adding the finitely many images $T_f^k(\tau)$ \emph{not} contained in $\sigma$ and then subdividing two dimensional cones appropriately.  After we do this for each $\tau\in\Delta$, we arrive at a fan $\Delta'$ such that
\begin{itemize}
 \item for any two dimensional $\sigma\in\Delta'$, either $T_f$ contracts $\sigma$ or $T_f^{-1}(\sigma)\subset \sigma'$ for some other two dimensional $\sigma\in\Delta'$.
 \item for any ray $\tau\in\Delta'$, either $T_f$ fixes $\tau$, or there exists $k\in\N$ such that $T_f^j(\tau) \in \Delta'$ for $j\leq k$ and $T_f^j(\tau) \in \sigma$ for all $j>k$ for some two dimensional $\sigma\in\Delta'$ such that $T_f(\sigma)\subset \sigma$.
\end{itemize}
The second property means that if $X'$ is the toric surface associated to $\Delta'$ and $C_{X'}$ is an irreducible curve in the complement of $(\C^*)^2$, either $f_{X'}(C_{X'}) = C_{X'}$, or there exists some $k\in\N$ such that $T_f^j$ contracts $C_{X'}$ if and only if $j>k$.  In particular, $f_{X'}$ is stable along $\eta_{X'}$.

The problem is that $X'$ is not (necessarily) smooth.  We fix this inductively. The  two dimensional cones of
$\Delta'$ that are expanded or contracted by $T_f$ are also cones of $\Delta$.  Since the surface $X$ associated to $\Delta$ is smooth, these cones are regular.  So if $\sigma\in\Delta'$ is irregular, we have for each $j>0$ that $T_f^{-k}(\sigma)$ is contained in a cone of $\Delta'$ different from $\sigma$.  We may suppose inductively that these (finitely many) other cones are all regular.  Thus we can regularize $\sigma$ by adding to $\Delta'$ finitely many rays in $\sigma$ and then subdividing $\sigma$.  Then we add forward images of these rays and subdivide the forward images of $\sigma$ as in the previous paragraph in order to maintain stability along $\eta$.  As we do this, we may also include more rays in order to ensure that each forward image $T^j(\sigma)$ is either a union of regular cones or falls into a contracted cone.  This does not change the cones that are expanded or contracted by $T_f$.  Nor does it affect regularity of the cones containing preimages of $\sigma$.
Rather, it reduces the number of irregular cones by at least one.  So in finitely many steps, $\Delta'$ becomes the fan for a smooth toric surface $X'$, such that $f_{X'}$ is stable along $\eta_{X'}$.  This proves that $f$ is corrigible along $\eta$.
\qed

\section*{Appendix: toric surfaces.}

Here we assemble needed facts about toric surfaces.  Our presentation is brief and partial, but the interested reader can consult \cite{CLS,Fu,Oda} for more detail.

A toric surface is a complex surface $X$ together with an embedding $(\C^*)^2\hookrightarrow X$ such that the natural action of the algebraic torus $(\C^*)^2$ on itself extends holomorphically to all of $X$.  The complement of $(\C^*)^2$ in $X$ is a finite union of smooth rational curves $C_1, \dots , C_p$.  These correspond to rational rays $\tau_j\subset \R^2$ as follows.

For any pair $u = (a,b) \in\Z^2\setminus\{\origin\}$, let $\gamma^u:\C^*\to (\C^*)^2$ denote the one parameter subgroup
$\gamma^u(z) = (z^a, z^b)$ .  Then for each $j=1,\dots,p$ there exists a unique relatively prime pair $u_j=(a_j,b_j)$ such that $\lim_{z\to 0} \gamma^{u_j}(z) \in C_j$.  Indeed for any integer $k>0$, the `landing' map $w \mapsto \lim_{\zeta\to 0} w\gamma^{ku_j}(\zeta)$ is a holomorphic map of $(\C^*)^2$ onto $C_j \setminus \bigcup_{j\neq k} C_k$.   For each $j$, let $\tau_j := \R_+u_j\subset\R^2$ denote the rational ray generated by $u_j$.  We choose indices so that the $\tau_j$ are ordered counterclockwise about $\origin$ and adopt the `mod $p$' convention $\tau_0 := \tau_p$.  

The surface $X$ is compact if $\lim_{z\to 0} \gamma^u(z)$ exists \emph{for all} pairs $u\in\Z^2$.  More precisely, each pair $\tau_j,\tau_{j+1}$ bounds a distinct closed and strictly convex cone $\sigma_j\subset\R^2$ and these cones cover $\R^2$.  A pair $u\in\Z^2$ lies in the interior of $\sigma_j$ if and only if $\lim_{z\to 0} \gamma^u(z)$ is the unique intersection between $C_j$ and $C_{j+1}$.
That $X$ is smooth means that each cone $\sigma_j$ is \emph{regular}, i.e. that $u_j,u_{j+1}$ give an (integral) basis for $\Z^2$.  The \emph{fan} of $X$ is the collection
$$
\Delta := \{\origin,\tau_1,\dots,\tau_p,\sigma_1,\dots,\sigma_p\}
$$
of $0$, $1$, and $2$ dimensional cones associated to $X$.

The blowup of $X$ at $C_j\cap C_{j+1}$ is also a toric surface whose fan is obtained by adding the ray generated by $u_j + u_{j+1}$ to the rays $\tau_j\in\Delta$ (and replacing the two dimensional cone $\sigma_j$ by the corresponding pair of subcones).  More generally, any smooth toric surface $Y$ whose fan contains all rays $\tau_j\in\Delta$ is a modification of $X$ that decomposes into a finite sequence of such blowups.  Every rational ray in $\R^2$ can be realized as an element in the fan of \emph{some} toric modification of $X$.  Hence any two toric surfaces $X,Y$ have a common toric modification $\Gamma\to X,Y$ with fan obtained by first joining the fans of $X$ and $Y$ and then adding more rays to ensure smoothness of $\Gamma$.  This common modification gives us a birational map $X\to Y$ which is canonical in
the sense that it respects the given embeddings of $(\C^*)^2$ into $X$ and $Y$.

Finally, we note that the form $\eta = \frac{dx\wedge dy}{xy}$, which is holomorphic on $(\C^*)^2$, extends to a meromorphic two-form $\eta_X$ on $X$ with $\div\eta_X = - \sum C_j$.  Hence $-\div\eta$ is reduced and effective with dual graph equal to a cycle.

\bibliographystyle{mjo}

\bibliography{refs}

\end{document}